\documentclass[a4paper,twoside]{article}

\def\shorttitle{Inverse Source Problem for $($Time-Fractional$)$ Diffusion Equations}
\def\shortauthor{Y. Liu and Z. Zhang}

\usepackage{amsfonts}
\usepackage{amsmath}
\usepackage{amssymb}
\usepackage{amscd}
\usepackage{amsbsy}
\usepackage{amsthm}
\usepackage{bm}
\usepackage{empheq}
\usepackage{graphicx}
\usepackage{psfrag}
\usepackage{color}
\usepackage{cite}
\usepackage{multicol}
\usepackage[normalem]{ulem}
\usepackage{epsfig,subfigure,epstopdf}

\allowdisplaybreaks[4]

\newfont{\myfnt}{cmssi10 scaled 1440}
\numberwithin{equation}{section}
\catcode`@=11
\def\ps@nk{\def\@oddhead{\vbox{\hbox to \hsize{\pic \footnotesize \it \shorttitle
\hfill \rm \thepage} \vspace{1mm} \vspace*{-2mm}}}
\def\@evenhead{\vbox{\hbox to \hsize{\pic \footnotesize \rm \thepage \hfill \it \shortauthor}
\vspace{1mm} \vspace*{-2mm}}}
\def\@oddfoot{} \def\@evenfoot{}}
\def\ps@first{\def\@oddhead{\vbox{\hbox to \hsize{\pic \footnotesize
} \break}}
\def\@oddfoot{} \def\@evenfoot{}}
\newtheoremstyle{thmstyle}% name
  {6pt}%      Space above
  {6pt}%      Space below
  {\it}%         Body font
  {}%         Indent amount (empty = no indent, \parindent = para indent)
  {\bf}% Thm head font
  {}%        Punctuation after thm head
  {.5em}%     Space after thm head: " " = normal interword space;
  {}%         Thm head spec (can be left empty, meaning `normal')
\newtheoremstyle{remstyle}% name
  {6pt}%      Space above
  {6pt}%      Space below
  {\rm}%         Body font
  {}%         Indent amount (empty = no indent, \parindent = para indent)
  {\bf}% Thm head font
  {}%        Punctuation after thm head
  {.5em}%     Space after thm head: " " = normal interword space;
  {}%         Thm head spec (can be left empty, meaning `normal')
\def\Section#1{\Sec{\large #1} \setcounter{equation}{0} \vskip -6mm \indent}
\def\Sec{\@Startsection{section}{1}{\z@}
                                   {-3.5ex \@plus -1ex \@minus -.2ex}%
                                   {2.3ex \@plus.2ex}%
                                   {\normalfont\large\bfseries\boldmath}}
\def\@Startsection#1#2#3#4#5#6{%
  \if@noskipsec \leavevmode \fi
  \par
  \@tempskipa #4\relax
  \@afterindenttrue
  \ifdim \@tempskipa <\z@
    \@tempskipa -\@tempskipa \@afterindentfalse
  \fi
  \if@nobreak
    \everypar{}%
  \else
    \addpenalty\@secpenalty\addvspace\@tempskipa
  \fi
  \@ifstar
    {\@ssect{#3}{#4}{#5}{#6}}%
    {\@dblarg{\@Sect{#1}{#2}{#3}{#4}{#5}{#6}}}}
\def\@Sect#1#2#3#4#5#6[#7]#8{%
  \ifnum #2>\c@secnumdepth
    \let\@svsec\@empty
  \else
    \refstepcounter{#1}%
    \protected@edef\@svsec{\@seccntformat{#1}\relax}%
  \fi
  \@tempskipa #5\relax
  \ifdim \@tempskipa>\z@
    \begingroup
      #6{%
          \@hangfrom{\hskip #3\relax\@svsec \hskip -2.5mm}%
          \interlinepenalty \@M #8\@@par}
    \endgroup
    \csname #1mark\endcsname{#7}%
    \addcontentsline{toc}{#1}{%
      \ifnum #2>\c@secnumdepth \else
        \protect\numberline{\csname the#1\endcsname}%
      \fi
      #7}%
  \else
    \def\@svsechd{%
      #6{\hskip #3\relax
      \@svsec #8}%
      \csname #1mark\endcsname{#7}%
      \addcontentsline{toc}{#1}{%
        \ifnum #2>\c@secnumdepth \else
          \protect\numberline{\csname the#1\endcsname}%
        \fi
        #7}}%
  \fi
  \@xsect{#5}}
\renewenvironment{abstract}{%
        \small
        \quotation
         \noindent {\bfseries \abstractname } }%
      {\if@twocolumn\else\endquotation\fi}

\def\Subsec{\@StartSubsection{subsection}{2}{\z@}%
                                     {-3.25ex\@plus -1ex \@minus -.2ex}%
                                     {1.5ex \@plus .2ex}%
                                     {\normalfont\normalsize\bfseries\boldmath}}
\def\@StartSubsection#1#2#3#4#5#6{%
  \if@noskipsec \leavevmode \fi
  \par
  \@tempskipa #4\relax
  \@afterindenttrue
  \ifdim \@tempskipa <\z@
    \@tempskipa -\@tempskipa \@afterindentfalse
  \fi
  \if@nobreak
    \everypar{}%
  \else
    \addpenalty\@secpenalty\addvspace\@tempskipa
  \fi
  \@ifstar
    {\@ssect{#3}{#4}{#5}{#6}}%
    {\@dblarg{\@SubSect{#1}{#2}{#3}{#4}{#5}{#6}}}}
\def\@SubSect#1#2#3#4#5#6[#7]#8{%
  \ifnum #2>\c@secnumdepth
    \let\@svsec\@empty
  \else
    \refstepcounter{#1}%
    \protected@edef\@svsec{\@seccntformat{#1}\relax}%
  \fi
  \@tempskipa #5\relax
  \ifdim \@tempskipa>\z@
    \begingroup
      #6{%
          \@hangfrom{\hskip #3\relax\@svsec\hskip -1.5mm}%
          \interlinepenalty \@M #8\@@par}
    \endgroup
    \csname #1mark\endcsname{#7}%
    \addcontentsline{toc}{#1}{%
      \ifnum #2>\c@secnumdepth \else
        \protect\numberline{\csname the#1\endcsname}%
      \fi
      #7}%
  \else
    \def\@svsechd{%
      #6{\hskip #3\relax
      \@svsec #8}%
      \csname #1mark\endcsname{#7}%
      \addcontentsline{toc}{#1}{%
        \ifnum #2>\c@secnumdepth \else
          \protect\numberline{\csname the#1\endcsname}%
        \fi
        #7}}%
  \fi
  \@xsect{#5}}
\def\list#1#2{\ifnum \@listdepth >5\relax \@toodeep \else \global
\advance \@listdepth\@ne \fi \rightmargin \z@ \listparindent\z@
\itemindent\z@ \csname @list\romannumeral\the\@listdepth\endcsname
\def\@itemlabel{#1}\let\makelabel\@mklab \@nmbrlistfalse #2\relax
\@trivlist \parskip 0pt \parindent\listparindent \advance \linewidth
-\rightmargin \advance\linewidth -\leftmargin \advance\@totalleftmargin
\leftmargin \parshape \@ne \@totalleftmargin \linewidth \ignorespaces}
\renewcommand{\@makecaption}[2]{\begin{center}#1. #2\end{center}}
\catcode`@=12 
\theoremstyle{thmstyle}
\newtheorem{thm}{\indent Theorem}[section]
\newtheorem{lem}[thm]{\indent Lemma}

\newtheorem{prob}[thm]{\indent Problem}
\theoremstyle{remstyle}

\newtheorem{algo}[thm]{\indent Algorithm}

\newsavebox{\mygraphic}
\def\pic{\begin{picture}(0,0) \put(-210,-1250){\usebox{\mygraphic}} \end{picture}}
\newfont{\HUGEbf}{cmbx10 scaled 3500}
\definecolor{gray}{rgb}{0.9,0.9,0.9}
\def\thebibliography#1{\section*{\bf \large References}
\list{[\arabic{enumi}]} {\settowidth \labelwidth{[#1]} \leftmargin
\labelwidth \advance \leftmargin \labelsep \usecounter{enumi}}
\def\newblock{\hskip .11em plus .33em minus .07em} \footnotesize \sloppy \clubpenalty
4000 \widowpenalty 4000 \sfcode`\.=1000 \relax}

\def\BC{\mathbb C}

\def\BR{\mathbb R}

\def\cD{\mathcal D}

\def\cU{\mathcal U}

\def\rd{\mathrm d}

\def\e{\mathrm e}

\def\supp{\mathrm{supp}}

\def\Ga{\Gamma}

\def\Om{\Omega}
\def\al{\alpha}
\def\be{\beta}
\def\ga{\gamma}
\def\de{\delta}
\def\ep{\epsilon}
\def\ve{\varepsilon}

\def\ze{\zeta}
\def\ka{\kappa}
\def\la{\lambda}
\def\si{\sigma}
\def\vp{\varphi}
\def\om{\omega}
\def\f{\frac}
\def\nb{\nabla}
\def\ov{\overline}
\def\pa{\partial}

\def\wt{\widetilde}
\def\tri{\triangle}
\def\beqnx{\begin{eqnarray*}} \def\eqnx{\end{eqnarray*}}

\theoremstyle{definition}

\numberwithin{equation}{section}

\title{\Large\bf\boldmath Reconstruction of the Temporal Component in the Source Term of a (Time-Fractional) Diffusion Equation}

\author{\large Yikan LIU$^*$\qquad Zhidong ZHANG$^\dag$}

\date{}

\begin{document}

\maketitle

\thispagestyle{first}
\renewcommand{\thefootnote}{\fnsymbol{footnote}}

\footnotetext{\hspace*{-5mm} \begin{tabular}{@{}r@{}p{14cm}@{}} &
Manuscript last updated: \today.\\
$^*$ & Graduate School of Mathematical Sciences, The University of Tokyo, 3-8-1 Komaba, Meguro-ku, Tokyo 153-8914, Japan. E-mail: ykliu@ms.u-tokyo.ac.jp\\
$^\dag$ & Corresponding author. Department of Mathematics, Texas A\&M University, College Station, TX 77843, USA. E-mail: zhidong@math.tamu.edu
\end{tabular}}

\renewcommand{\thefootnote}{\arabic{footnote}}

\begin{abstract}
In this article, we consider the reconstruction of $\rho(t)$ in the (time-fractional) diffusion equation $(\pa_t^\al-\tri)u(x,t)=\rho(t)g(x)$ ($0<\al\le1$) by the observation at a single point $x_0$. We are mainly concerned with the situation of $x_0\not\in\supp\,g$, which is practically important but has not been well investigated in literature. Assuming the finite sign changes of $\rho$ and an extra observation interval, we establish the multiple logarithmic stability for the problem based on a reverse convolution inequality and a lower estimate for positive solutions. Meanwhile, we develop a fixed point iteration for the numerical reconstruction and prove its convergence. The performance of the proposed method is illustrated by several numerical examples.

\vskip 4.5mm

\noindent\begin{tabular}{@{}l@{ }p{10cm}} {\bf Keywords } &
Fractional diffusion equation, Inverse source problem,\\
& Multiple logarithmic stability, Reverse convolution inequality,\\
& Fixed point iteration
\end{tabular}

\vskip 4.5mm


\end{abstract}

\baselineskip 14pt

\setlength{\parindent}{1.5em}

\setcounter{section}{0}

\Section{Introduction}\label{sec-intro}

Let $0<\al\le1$ and $\pa_t^\al$ denote the Caputo derivative defined as
\begin{equation}\label{eq-def-Caputo}
\pa_t^\al f(t):=\left\{\!\begin{alignedat}{2}
& \f1{\Ga(1-\al)}\int_0^t\f{f'(s)}{(t-s)^\al}\,\rd s, & \quad & 0<\al<1,\\
& f'(t), & \quad & \al=1,
\end{alignedat}\right.
\end{equation}
where $\Ga(\,\cdot\,)$ is the usual Gamma function. Let $d=1,2,3$ be the spatial dimensions and $\Om\subset\BR^d$ be an open bounded domain with a smooth boundary (e.g., of $C^2$-class). In this paper, we consider the Cauchy problem for a (time-fractional) diffusion equation
\begin{equation}\label{eq-ivp-u}
\begin{cases}
(\pa_t^\al-\tri)u(x,t)=\rho(t)g(x), & x\in\BR^d,\ 0<t<T,\\
u(x,0)=0, & x\in\BR^d
\end{cases}
\end{equation}
as well as the corresponding initial-boundary value problem
\begin{equation}\label{eq-ibvp-u}
\begin{cases}
(\pa_t^\al-\tri)u(x,t)=\rho(t)g(x), & x\in\Om,\ 0<t<T,\\
u(x,0)=0, & x\in\Om,\\
u(x,t)=0, & x\in\pa\Om,\ 0<t<T.
\end{cases}
\end{equation}
Here $\tri:=\sum_{j=1}^d\pa_{x_j}^2$ denotes the Laplacian in space. The conditions on the temporal component $\rho$ and the spatial component $g$ in the source term will be specified later.

With $\al=1$, the governing equation in \eqref{eq-ivp-u} and \eqref{eq-ibvp-u} is well known as an orthodox model for diffusion and thermal conduction processes. Since the 1990s, its fractional counterpart with $0<\al<1$ has witnessed its significance as a candidate for modeling the anomalous diffusion phenomena in heterogeneous media (see, e.g., \cite{GCR92,HH98}). From then on, the time-fractional parabolic operator $\pa_t^\al-\tri$ and its further variants have gathered increasing attentions within mathematicians. Here we do not intend to provide a complete list of bibliographies, but just refer to \cite{EK04,L09,SY11a,GLY15,LN16} which reveal the fundamental properties of time-fractional diffusion equations. It turns out that the fractional case resembles its integer prototype to a certain extent, whereas diverges remarkably in such senses as weaker smoothing effect and unique continuation property. Nevertheless, concerning the strong maximum principle, the strict positivity of the solution with realistic initial data is shared for any $0<\al\le1$ according to \cite{LRY16,JPY17}. As a direct application of the established maximum principle, the uniqueness of the following inverse source problem was also immediately proved.

\begin{prob}\label{prob-isp}
Let $u_\rho$ be the solution to \eqref{eq-ivp-u} or \eqref{eq-ibvp-u}, and fix $x_0\in\BR^d$ or $x_0\in\Om$ arbitrarily. Provided that the spatial component $g(x)$ in the source term is known, determine the temporal component $\rho(t)$ by the single point observation data $u_\rho(x_0,t)\ (0<t<T)$.
\end{prob}

Although \cite{LRY16,JPY17} only dealt with the initial-boundary value problem \eqref{eq-ibvp-u} with $0<\al<1$, the same argument simply works for all other cases, and we omit the details here. As natural sequels after the uniqueness, we are interested in the theoretical stability as well as the numerical treatment of Problem \ref{prob-isp}, which are the main focuses of this paper.

As was explained in \cite{LRY16,L17}, the formulation of Problem \ref{prob-isp} stands for the identification of the time evolution pattern $\rho(t)$ of the (anomalous) diffusion by the observation data taken at a single monitoring point $x_0$. The time-independent component $g(x)$ of the source term models the spatial distribution of the contaminant source, which is assumed to be given. Such a problem setting is applicable e.g.\! for the Chernobyl disaster and the Fukushima Daiichi nuclear disaster, in which the location of the source is exactly known, but the decay of the radiative strength in time is unknown. In these cases, the contaminant source is harmful and it is extremely dangerous or even impossible to measure inside the support of $g$. Therefore, a more desirable situation in practice is to carry out the observation far away from the location of the source, that is, $x_0\not\in\supp\,g$.

Unfortunately, the above practical setting causes considerable difficulties in mathematical analysis, and the related literature is rather limited even in the case of $\al=1$. Actually, if $x_0\in\supp\,g$ or simply $g(x_0)\ne0$, one can reduce Problem \ref{prob-isp} into a Volterra equation of the second kind, by which the unique solvability follows immediately (see \cite[\S 1.5]{POV00}). In the case of $x_0\not\in\supp\,g$, Cannon and P\'erez Esteva \cite{CP86} proved the following logarithmic stability
\[\|\rho\|_{C[0,T]}\le C\left|\log\|u_\rho(x_0,\,\cdot\,)\|_{L^2(0,\infty)}\right|^{-2}\]
for $\|u_\rho(x_0,\,\cdot\,)\|_{L^2(0,\infty)}<1$ and $d=1,2$, where the known component $g$ was restricted as the characterization function $\chi_\om$ of a subdomain $\om\subset\Om$. However, the data should be taken in the unrealistic $(0,\infty)$ interval. By assuming the finite sign changes of $\rho$, Saitoh, Tuan and Yamamoto \cite{STY02} asserted the H\"older stability for Problem \ref{prob-isp} with a small extra interval $(T,T+\de)$ of observation data. However, there seems to be several fatal flaws in the arguments in \cite{STY02}, so that the essential ill-posedness of the problem was much underestimated. Following the same line, in this paper we attempt to fill the gaps in \cite{STY02} and generalize the result to the fractional case. Regarding other inverse problems on determining time-dependent coefficients in parabolic equations, we refer to \cite{C84,CK13}.

Concerning inverse problems for time-fractional diffusion equations, publications for time-dependent unknown coefficients are rather minor compared with that for space-dependent ones. On Problem \ref{prob-isp}, Sakamoto and Yamamoto \cite{SY11a} proved
\begin{equation}\label{eq-stab-0}
\|\rho\|_{C[0,T]}\le C\|\pa_t^\al u_\rho(x_0,\,\cdot\,)\|_{C[0,T]}
\end{equation}
by assuming $g(x_0)\ne0$. Allowing $g=g(x,t)$ with $|g(x_0,\,\cdot\,)|\ge g_0$ a.e.\! in $(0,T)$ for a constant $g_0>0$, Fujishiro and Kian \cite{FK16} obtained the similar Lipschitz stability as \eqref{eq-stab-0} in the $L^p(0,T)$ norm. Recently, a numerical reconstruction method for $\rho$ was proposed in \cite{WLL16} by the partial boundary measurements. It turns out that all the above mentioned results rely heavily on some non-vanishing assumptions such as $x_0\in\supp\,g$, and the requirements on the data regularity were quite strong. In the case of $x_0\not\in\supp\,g$, only \cite{LRY16,L17} showed the uniqueness by the strong maximum principle, but there seems no published results on the stability and the numerical inversion to the best knowledge of the authors.

The present paper is mainly motivated by the settings in \cite{STY02}, namely, the assumption on finite sign changes of $\rho$ and an extra interval of observation data. As a correction to the flaws therein, we prove a new reverse convolution inequality (see Lemma \ref{lem-rci}), which is the key to the stability. Then we employ the fundamental solutions of (time-fractional) diffusion equations to give lower estimates of positive solutions to homogeneous problems, which enables us to establish the extremely weak stability of multiple logarithmic type (see Theorem \ref{thm-stab-log}). Numerically, we develop a fixed point iteration for the numerical reconstruction, and provide a convergence result.

The rest of this paper is organized as follows. Recalling some facts in fractional calculus and (time-fractional) diffusion equations, in Section \ref{sec-pre} we state the main stability results in Theorems \ref{thm-stab} and \ref{thm-stab-log} concerning Problem \ref{prob-isp}. Then Section \ref{sec-proof} is devoted to the proofs of a key lemma and the above theorems. In Section \ref{sec-recon}, we propose an iteration method \eqref{eq-itr} for the numerical treatment of Problem \ref{prob-isp} and prove Theorem \ref{thm-cov} for its convergence. We implement several numerical examples in Section \ref{sec-numer} to illustrate the performance of the proposed method, and summarize the paper in Section \ref{sec-concl} with some concluding remarks.

\Section{Preliminaries and Main Results}\label{sec-pre}

We start from the introduction of general notations and terminologies. First, we introduce the usual $L^2(\Om)$ space with the inner product $(\,\cdot\,,\,\cdot\,)$ and the Sobolev spaces $H_0^1(\Om)$, $H^2(\Om)$, etc.. Next, let $\{(\la_n,\vp_n)\}_{n=1}^\infty$ be the eigensystem of the elliptic operator $-\tri:H^2(\Om)\cap H_0^1(\Om)\longrightarrow L^2(\Om)$ such that $0<\la_1<\la_2\le\cdots$, $\la_n\uparrow\infty$ ($n\to\infty$) and $\{\vp_n\}\subset H^2(\Om)\cap H_0^1(\Om)$ forms a complete orthonormal basis of $L^2(\Om)$. Then the fractional power $(-\tri)^\ga$ can be further induced for $\ga\ge0$ as (see, e.g., \cite{P83})
\[\cD((-\tri)^\ga)=\left\{f\in L^2(\Om);\,\sum_{n=1}^\infty|\la_n^\ga(f,\vp_n)|^2<\infty\right\},\quad(-\tri)^\ga f:=\sum_{n=1}^\infty\la_n^\ga(f,\vp_n)\vp_n,\]
and $\cD((-\tri)^\ga)$ is a Hilbert space equipped with the norm
\[\|f\|_{\cD((-\tri)^\ga)}=\left(\sum_{n=1}^\infty|\la_n^\ga(f,\vp_n)|^2\right)^{\f12}.\]
Further, we know $\cD((-\tri)^\ga)\subset H^{2\ga}(\Om)$ for $\ga>0$ and especially $\cD((-\tri)^{\f12})=H_0^1(\Om)$.

For later use, we recall the Riemann-Liouville integral operator
\[J^\be f(t):=\left\{\!\begin{alignedat}{2}
& f(t), & \quad & \be=0,\\
& \f1{\Ga(\be)}\int_0^t\f{f(s)}{(t-s)^{1-\be}}\,\rd s, & \quad & 0<\be\le1,
\end{alignedat}\right.\quad f\in C[0,\infty).\]
Then by definition, the Caputo derivative $\pa_t^\be$ can be rewritten as $\pa_t^\be f=J^{1-\be}(f')$ for $0<\be\le1$. Further, for $f\in C^1[0,\infty)$ the Riemann-Liouville derivative is defined by
\begin{equation}\label{eq-def-RL}
D_t^\be f(t):=(J^{1-\be}f)'(t)=\left\{\!\begin{alignedat}{2}
& \f1{\Ga(1-\be)}\f\rd{\rd t}\int_0^t\f{f(s)}{(t-s)^\be}\rd s, & \quad & 0\le\be<1,\\
& f'(t), & \quad & \be=1.
\end{alignedat}\right.
\end{equation}
The following lemma collects basic properties of fractional derivatives, which can be simply verified by direct calculations.

\begin{lem}\label{lem-Caputo-RL}
Let $0<\be<1$ and $f\in C^1[0,\infty)$. For the Caputo derivative $\pa_t^\be$ and the Riemann-Liouville derivative $D_t^\be$ defined in \eqref{eq-def-Caputo} and \eqref{eq-def-RL} respectively, the following formulae hold:
\[D_t^\be f(t)=\f{f(0)}{\Ga(1-\be)}t^{-\be}+\pa_t^\be f(t),\quad f=J^\be(D_t^\be f),\quad f'=D_t^{1-\be}(D_t^\be f).\]
Especially, $D_t^\be f=\pa_t^\be f$ if $f(0)=0$. Moreover, if $f(0)=\pa_t^\be f(0)=0$, then
\[\pa_t^{\be'}\pa_t^\be f=\pa_t^{\be+\be'}f\quad\mbox{for }0\le\be'\le1-\be.\]
\end{lem}

\begin{lem}[Fractional Duhamel's principle]\label{lem-Duhamel}
Let $0<\al\le1$ and assume $\rho\in C^1[0,\infty)$.

{\rm(a)} Let $u_\rho$ be the solution to \eqref{eq-ivp-u}, where we assume
\begin{equation}\label{eq-asp-g1}
\begin{cases}
g:\mbox{bounded continuous} & \mbox{if }d=1,\\
g:\mbox{locally H\"older continuous} &  \mbox{if }d=2,3.
\end{cases}
\end{equation}
Then $u_\rho$ allows the representation
\begin{equation}\label{eq-Duhamel}
u_\rho(x,t)=\int_0^tD_s^{1-\al}\rho(s)\,v(x,t-s)\,\rd s,\quad t>0,
\end{equation}
where $v$ solves the homogeneous problem
\begin{equation}\label{eq-ivp-v}
\begin{cases}
(\pa_t^\al-\tri)v=0 & \mbox{in }\BR^d\times(0,\infty),\\
v=g & \mbox{in }\BR^d\times\{0\}.
\end{cases}
\end{equation}

{\rm(b)} Let $u_\rho$ be the solution to \eqref{eq-ibvp-u}, where we assume
\begin{equation}\label{eq-asp-g2}
g\in\cD((-\tri)^\ve)\mbox{ with some }\ve>0.
\end{equation}
Then $u_\rho$ also allows the representation \eqref{eq-Duhamel}, where $v$ solves the homogeneous problem
\begin{equation}\label{eq-ibvp-v}
\begin{cases}
(\pa_t^\al-\tri)v=0 & \mbox{in }\Om\times(0,\infty),\\
v=g & \mbox{in }\Om\times\{0\},\\
v=0 & \mbox{on }\pa\Om\times(0,\infty).
\end{cases}
\end{equation}
\end{lem}

For integers $\al=1,2,\ldots$, the above lemma is a fundamental property for evolution equations. For $0<\al<1$, there appears the Riemann-Liouville derivative of the temporal component $\rho$ in the convolution \eqref{eq-Duhamel}, which automatically results from direct calculations. On the other hand, conditions \eqref{eq-asp-g1} and \eqref{eq-asp-g2} on the spatial component $g$ are technical and we refer to \cite{EK04,LRY16} for further details.

Now that the inhomogeneous problems \eqref{eq-ivp-u} and \eqref{eq-ibvp-u} are related to the homogeneous ones \eqref{eq-ivp-v} and \eqref{eq-ibvp-v} respectively via Lemma \ref{lem-Duhamel}, we shall also recall some basic properties of the homogeneous problems for later use. First we start with the whole space case \eqref{eq-ivp-v}.

\begin{lem}[Eidelman and Kochubei \cite{EK04}]\label{lem-ivp}
Let $g$ satisfy condition \eqref{eq-asp-g1}. Then there exists a classical solution to the Cauchy problem \eqref{eq-ivp-v} which takes the form
\begin{equation}\label{eq-ivp-fdm}
v(x,t)=\int_{\BR^d}K_\al(x-x',t)\,g(x')\,\rd x',\quad0<\al\le1,
\end{equation}
where the fundamental solution $K_\al(x,t)$ satisfies the following properties.

{\rm(a)} If $0<|x|^2\le t^\al$, then there exists a constant $c_1>0$ depending on $\al,d$ such that
\begin{equation}\label{eq-Fox-est0}
|K_\al(x,t)|\le\begin{cases}
c_1\,t^{-\f\al2}, & d=1,\\
c_1\,t^{-\al}(|\log(t^{-\al}|x|^2)|+1), & d=2,\\
c_1\,t^{-\al}|x|^{-1}, & d=3.
\end{cases}
\end{equation}

{\rm(b)} If $|x|^2\ge t^\al$, then there exist constants $c_1>0$ and $c_2>0$ depending on $\al,d$ such that
\begin{equation}\label{eq-Fox-est1}
|K_\al(x,t)|\le c_1\,t^{-\f{\al d}2}\exp\left(-c_2\,t^{-\f\al{2-\al}}|x|^{\f2{2-\al}}\right).
\end{equation}

{\rm(c)} If $|x|>r$ for some fixed $r>0$, then there exist constants $c_3>0$ and $c_4>0$ depending on $\al,d,r$ such that the following asymptotic behavior holds for $t\downarrow0$:
\begin{equation}\label{eq-Fox-asp}
K_\al(x,t)\sim c_3\,t^{-\f{\al d}{2(2-\al)}}|x|^{-\f{d(1-\al)}{2-\al}}\exp\left(-c_4\,t^{-\f\al{2-\al}}|x|^{\f2{2-\al}}\right).
\end{equation}
\end{lem}

The fundamental function $K_\al(x,t)$ involves the Fox $H$-function (see \cite{MSH09}), which takes a rather complicated form. For conciseness, we will not scrutinize the $H$-function in detail, but just exploit some of its properties such as \eqref{eq-Fox-est1} and \eqref{eq-Fox-asp}. Moreover, $K_\al(x,t)$ generalizes the usual heat kernel with $\al=1$, i.e.,
\[K_1(x,t)=(4\pi\,t)^{-\f d2}\exp\left(-\f{|x|^2}{4\,t}\right).\]

To express the solution to the initial-boundary value problem \eqref{eq-ibvp-v}, we recall the Mittag-Leffler function defined by
\[E_{\al,\be}(z):=\sum_{k=0}^\infty\f{z^k}{\Ga(\al k+\be)},\quad z\in\BC,\ \al>0,\ \be\in\BR.\]
If $0<\al<2$, then there exists a constant $c_0>0$ depending only on $\al,\be$ such that
\begin{equation}\label{eq-est-ML}
|E_{\al,\be}(-\eta)|\le\f{c_0}{1+\eta},\quad\forall\,\eta\ge0.
\end{equation}
Now we proceed to the bounded domain case \eqref{eq-ibvp-v}.

\begin{lem}\label{lem-ibvp}
{\rm(a)} Let $g\in L^2(\Om)$. Then there exists a unique solution $v\in L^1(0,T;H^2(\Om)\cap H_0^1(\Om))$ to the initial-boundary value problem \eqref{eq-ibvp-v} which takes the form
\begin{equation}\label{eq-ibvp-sol}
v(\,\cdot\,,t)=\sum_{n=1}^\infty E_{\al,1}(-\la_nt^\al)(g,\vp_n)\vp_n.
\end{equation}

{\rm(b)} Especially, for $\al=1$, $v$ also allows the representation
\begin{align}
v(x,t) & =(4\pi\,t)^{-\f d2}\int_\Om\exp\left(-\f{|x-x'|^2}{4\,t}\right)g(x')\,\rd x'\nonumber\\
& \quad\,+\int_0^t(4\pi(t-s))^{-\f d2}\int_{\pa\Om}\exp\left(-\f{|x-y|^2}{4(t-s)}\right)\pa_\nu v(y,s)\,\rd y\rd s,\label{eq-ibvp-fdm}
\end{align}
where $\pa_\nu v:=\nb v\cdot\nu$ is the normal derivative and $\nu=\nu(y)$ denotes the unit outward normal vector at $y\in\pa\Om$.

{\rm(c)} If the spatial dimension $d=1$ and $g\in H_0^1(\Om)$, then $\pa_tv(x_0,\,\cdot\,)\in L^1(0,T)$ for any $x_0\in\Om$.
\end{lem}

\begin{proof}
The conclusions in (a) are summarized from Sakamoto and Yamamoto \cite{SY11a}, and the formula \eqref{eq-ibvp-fdm} in (b) follows from direct calculations and integration by parts. To show (c), we employ the formula $\f\rd{\rd t}E_{\al,1}(-\la_nt^\al)=-\la_nt^{\al-1}E_{\al,\al}(-\la_nt^\al)$ and differentiate \eqref{eq-ibvp-sol} to formally write
\[\pa_tv(\,\cdot\,,t)=-t^{\al-1}\sum_{n=1}^\infty \la_nE_{\al,\al}(-\la_nt^\al)(g,\vp_n)\vp_n,\quad t>0.\]
Since $g\in H_0^1(\Om)=\cD((-\tri)^{\f12})$, we fix any $\ga\in(\f14,\f12)$ and use \eqref{eq-est-ML} to estimate
\begin{align*}
\|\pa_tv(\,\cdot\,,t)\|_{\cD((-\tri)^\ga)}^2 & =t^{2(\al-1)}\sum_{n=1}^\infty\left|\la_n^{\ga+\f12}E_{\al,\al}(\la_nt^\al)\la_n^{\f12}(g,\vp_n)\right|^2\\
& \le t^{\al(1-2\ga)-2}\sum_{n=1}^\infty\f{c_0(\la_nt^\al)^{\ga+\f12}}{1+\la_nt^\al}\left|\la_n^{\f12}(g,\vp_n)\right|^2\le\left(c_0\|g\|_{H^1(\Om)}t^{\al(\f12-\ga)-1}\right)^2.
\end{align*}
By the one-dimensional Sobolev embedding $\cD((-\tri)^\ga)\subset C(\ov\Om)$ with $\ga>\f14$, there exists a constant $c_\ga>0$ such that
\[|\pa_tv(x_0,\,\cdot\,)|\le\|\pa_tv(\,\cdot\,,t)\|_{C(\ov\Om)}\le c_\ga c_0\|g\|_{H^1(\Om)}t^{\al(\f12-\ga)-1},\quad t>0,\]
indicating $\pa_tv(x_0,\,\cdot\,)\in L^1(0,T)$ immediately.
\end{proof}

Next we prepare for the statement of the main theorems. Let $0<\al\le1$ and $T>0$. For given constants $M>0$ and $N=0,1,2,\ldots$, define the admissible set
\begin{align}
\cU_N:=\{f\in C^1[0,T];\; & \|f\|_{C^1[0,T]}\le M,\nonumber\\
& f\mbox{ changes signs at most }N\mbox{ times on }(0,T)\}.\label{eq-def-sign}
\end{align}
The restriction on the number of sign changes follows the same line as that in \cite{STY02}, where the same problem as Problem \ref{prob-isp} was investigated for heat equations. Nevertheless, here we assume $\rho\in C^1[0,T]$ for $D_t^{1-\al}\rho$ to make sense in $L^1(0,T)$.

Now we state the first main result of this paper.

\begin{thm}\label{thm-stab}
Fix $0<\al\le1$, $T>0$ and let $\cU_N$ be defined as that in \eqref{eq-def-sign} with given constants $M>0$ and $N=0,1,2,\ldots$. Let $u_\rho$ be the solution to \eqref{eq-ivp-u} or \eqref{eq-ibvp-u}, where we assume that $g\ge0$, $g\not\equiv0$ and $g$ satisfies \eqref{eq-asp-g1} in the case of \eqref{eq-ivp-u} or \eqref{eq-asp-g2} in the case of \eqref{eq-ibvp-u}.

{\rm(a)} If $N=0$, then for any $\de\in(0,T)$, there exists a constant $B_\de>0$ depending only on $x_0,g$ such that $B_\de\uparrow\infty\ (\de\downarrow0)$ and
\begin{equation}\label{eq-stab-a}
\|\rho\|_{L^1(0,T-\de)}\le\f{T^{1-\al}B_\de}{\Ga(2-\al)}\|u_\rho(x_0,\,\cdot\,)\|_{L^1(0,T)},\quad\forall\,\rho\in\cU_0.
\end{equation}

{\rm(b)} If $N>0$, then there exist constants $C_0>0$ such that for any sufficiently small $\de>0$, there holds
\begin{equation}\label{eq-stab-b}
\|\rho\|_{L^1(0,T-\de)}\le\f{(C_0B_\de+1)^{N+1}-1}{C_0B_\de}\left(\f{T^{1-\al}B_\de}{\Ga(2-\al)}\|u_\rho(x_0,\,\cdot\,)\|_{L^1(0,T)}+M\de^2\right),\quad\forall\,\rho\in\cU_N,
\end{equation}
where $B_\de$ is the same constant as that in {\rm(a)}.
\end{thm}

By the assumptions in Theorem \ref{thm-stab} and Lemmas \ref{lem-ivp}--\ref{lem-ibvp}, we see that the observation data $u_\rho(x_0,\,\cdot\,)$ makes sense in $L^1(0,T)$. Parallel to that in \cite{STY02}, both estimates in Theorem \ref{thm-stab} are obtained at the cost of a small extra interval $(T-\de,T)$ of measurements. Such an interval cannot be arbitrarily small because the constant $B_\de\uparrow\infty$ as $\de\downarrow0$. Actually, the choice of $B_\de$ will be
\begin{equation}\label{eq-def-Bd}
B_\de:=\f1{\|v(x_0,\,\cdot\,)\|_{L^1(0,\de)}},\quad\de>0
\end{equation}
in the proof, where $v$ stands for the solution to the homogeneous problem \eqref{eq-ivp-v} or \eqref{eq-ibvp-v}.

Let us explain the results in Theorem \ref{thm-stab} in detail. First, in comparison with \cite{SY11a,FK16} where the same problem as Problem \ref{prob-isp} were treated, the restriction $x_0\in\supp\,g$ is removed, indicating that the monitoring point is allowed to be far away from the support of $g$. As was mentioned in Section \ref{sec-intro}, such a relaxation is significant in some practical applications. Moreover, estimates \eqref{eq-stab-a} and \eqref{eq-stab-b} only require the $L^1$-norm of the observation data, while that in \cite{SY11a,FK16} involves the Caputo derivative. Obviously, the former is more desirable in treating the noisy data.

Second, Theorem \ref{thm-stab}(a) gives a Lipschitz stability estimate \eqref{eq-stab-a} for the inverse source problem under the assumption that $\rho$ does not change its sign on $(0,T)$. In this case, the maximum principle (see Luchko \cite{L09}) indicates that $u_\rho(x_0,\,\cdot\,)$ also keeps the sign. This is unrealistic because $u_\rho(x_0,\,\cdot\,)$ stands for the measurement error in the context of stability analysis. Nevertheless, in Section \ref{sec-recon} we will propose an iterative method producing a monotone sequence converging to the true solution, where the stability of Lipschitz type becomes remarkable for the fast convergence.

Third, in the general case of $N=1,2,\ldots$, i.e., $\rho$ changes signs finite times on $(0,T)$, the estimate \eqref{eq-stab-b} in Theorem \ref{thm-stab}(b) looks complicated with some extra terms. Especially, the constant multiplier of $\|u_\rho(x_0,\,\cdot\,)\|_{L^1(0,T)}$ increases with larger $N$. In order to further treat the estimate \eqref{eq-stab-b}, according to \eqref{eq-def-Bd} we have to give a lower estimate of $\|v(x_0,\,\cdot\,)\|_{L^1(0,\de)}$ with respect to $\de$. This is possible in some cases and we collect the result as follows.

\begin{thm}\label{thm-stab-log}
Under the same conditions as that in Theorem $\ref{thm-stab}$, further assume that $N>0$ and $\|u_\rho(x_0,\,\cdot\,)\|_{L^1(0,T)}$ is sufficiently small.

{\rm(a)} If $x_0\in\supp\,g$, further assume $g\in C(\ov\Om)$ in the case of \eqref{eq-ibvp-u}. Then there exist a sufficiently small constant $\de_*>0$ and a constant $C_1>0$ such that
\begin{equation}\label{eq-stab-c}
\|\rho\|_{L^1(0,T-\de_*)}\le C_1\|u_\rho(x_0,\,\cdot\,)\|_{L^1(0,T)}^{(\f23)^{N+1}},\quad\forall\,\rho\in\cU_N.
\end{equation}

{\rm(b)} In the case of \eqref{eq-ivp-u}, additionally assume that $g$ is compactly supported and $x_0\not\in\ov{\supp\,g}$. In the case of \eqref{eq-ibvp-u}, additionally require $\al=1$ and
\begin{equation}\label{eq-asp-x0}
\ka_0:=\sup_{x\in\supp\,g}|x-x_0|\le\inf_{y\in\pa\Om}|y-x_0|=:\ka_1.
\end{equation}
Then there exist a sufficiently small $\de_*>0$ and a constant $C_2>0$ such that
\begin{equation}\label{eq-stab-d}
\|\rho\|_{L^1(0,T-\de_*)}\le C_2\,(\underbrace{\log(\cdots(\log|\log}_{N+1}\|u_\rho(x_0,\,\cdot\,)\|_{L^1(0,T)}|)\cdots))^{-2(\f2\al-1)},\quad\forall\,\rho\in\cU_N.
\end{equation}
\end{thm}

Obviously, estimates \eqref{eq-stab-c} and \eqref{eq-stab-d} indicate a weak H\"older stability and an extremely weak stability of the multiple logarithmic type, respectively. In this sense, Theorem \ref{thm-stab-log} reflects the severe ill-posedness of Problem \ref{prob-isp} as the observation data $u_\rho(x_0,\,\cdot\,)\to0$ in $L^1(0,T)$. Especially, the stability becomes worse with a larger number $N$ of the sign changes of $\rho$. Roughly speaking, both \eqref{eq-stab-c} and \eqref{eq-stab-d} result from a minimization process based on the estimate \eqref{eq-stab-b}, and the minimizer $\de_*>0$ more or less takes the place of an optimal choice of the regularization parameter.

In Theorem \ref{thm-stab-log}(a), we reconsider the situation of $x_0\in\supp\,g$ as the majority of related literature restricted. In this case, \cite[Theorem 4.4]{SY11a} and \cite[Theorem 1.1]{FK16} achieve stronger stabilities than our H\"older one \eqref{eq-stab-c} at the cost of higher regularity assumptions on the observation data.

Remarkably, Theorem \ref{thm-stab-log}(b) allows us to choose $x_0\not\in\supp\,g$ under certain technical conditions. Owing to the power $-2(\f2\al-1)$ in \eqref{eq-stab-d}, we observe that the stability is slightly improved with smaller $\al$, though such an improvement does not change the essential weakness of the stability. This indicates the similarity of the cases of $\al=1$ and $0<\al<1$ in our problem, unlike the significant difference occurring in many other problems. In fact, such a similarity originates from the exponential decay of $\|v(x_0,\,\cdot\,)\|_{L^1(0,\de)}$ as $\de\downarrow0$ for all $0<\al\le1$ under the conditions in Theorem \ref{thm-stab-log}. For this part, we have to take advantage of the fundamental solutions introduced before. Especially, in the bounded domain case, we require $\al=1$ because the fundamental solution for \eqref{eq-ibvp-v} with $0<\al<1$ is not yet available.

\Section{Proofs of Theorems \ref{thm-stab} and \ref{thm-stab-log}}\label{sec-proof}

In this section, we prove the above theorems based on the representation \eqref{eq-Duhamel} in Lemma \ref{lem-Duhamel}. Following the argument in the proof of \cite[Theorem 2.6]{JLLY17}, we take the Riemann-Liouville integral operator $J^{1-\al}$ in \eqref{eq-Duhamel} to deduce
\begin{equation}\label{eq-Duhamel-J}
J^{1-\al}u_\rho(\,\cdot\,,t)=\int_0^t\rho(s)\,v(\,\cdot\,,t-s)\,\rd s.
\end{equation}
Moreover, as long as $u_\rho(x_0,\,\cdot\,)\in L^1(0,T)$, Young's inequality immediately gives
\begin{align}
\|J^{1-\al}u_\rho(x_0,\,\cdot\,)\|_{L^1(0,T)} & =\f1{\Ga(1-\al)}\left\|\int_0^t\f{u_\rho(x_0,s)}{(t-s)^\al}\,\rd s\right\|_{L^1(0,T)}\nonumber\\
& \le\f{T^{1-\al}}{\Ga(2-\al)}\|u_\rho(x_0,\,\cdot\,)\|_{L^1(0,T)}.\label{eq-est-RL}
\end{align}
Therefore, the problem is now reduced to estimating $\rho$ in the convolution \eqref{eq-Duhamel-J} with $x=x_0$. To this end, the following lemma plays a fundamental role.

\begin{lem}[Reverse convolution inequality]\label{lem-rci}
Let $0\le\eta<T_0<\infty$ and $\de>0$ be arbitrarily given. Suppose that $f_1\in L^1(\eta,T_0+\de)$, $f_2\in L^1(0,T_0+\de-\eta)$ and $f_2\ge0$ on $(0,T_0+\de-\eta)$.

{\rm(a)} If $f_1$ keeps its sign on $(\eta,T_0+\de)$, then
\begin{equation}\label{eq-RCI-a}
\|f_1\|_{L^1(\eta,T_0)}\|f_2\|_{L^1(0,\de)}\le\left\|\int_\eta^tf_1(s)f_2(t-s)\,\rd s\right\|_{L^1(\eta,T_0+\de)}.
\end{equation}

{\rm(b)} If $f_1$ only keeps its sign on $(\eta,T_0)$, then
\begin{equation}\label{eq-RCI-b}
\|f_1\|_{L^1(\eta,T_0)}\|f_2\|_{L^1(0,\de)}\le\left\|\int_\eta^tf_1(s)f_2(t-s)\,\rd s\right\|_{L^1(\eta,T_0+\de)}+2\|f_1\|_{L^1(T_0,T_0+\de)}\|f_2\|_{L^1(0,\de)}.
\end{equation}
\end{lem}

\begin{proof}
We consider
\[I:=\int_\eta^{T_0+\de}\!\!\!\!\int_\eta^t|f_1(s)|\,f_2(t-s)\,\rd s\rd t.\]
Without loss of generality, we can assume $f_1\ge0$ on $(\eta,T_0)$, or otherwise we may consider $-f_1$ instead. Similarly to the proof of \cite[Theorem 3.1]{STY02}, we estimate $I$ from below as
\begin{align}
I & =\int_\eta^{T_0+\de}|f_1(s)|\int_s^{T_0+\de}f_2(t-s)\,\rd t\rd s=\int_\eta^{T_0+\de}|f_1(s)|\int_0^{T_0+\de-s}f_2(t)\,\rd t\rd s\nonumber\\
& \ge\int_\eta^{T_0}|f_1(s)|\int_0^{T_0+\de-s}f_2(t)\,\rd t\rd s\ge\int_\eta^{T_0}|f_1(s)|\int_0^\de f_2(t)\,\rd t\rd s\nonumber\\
& =\|f_1\|_{L^1(\eta,T_0)}\|f_2\|_{L^1(0,\de)}.\label{eq-RCI-1}
\end{align}

(a) If $f_1\ge0$ throughout $(\eta,T_0+\de)$, then obviously
\[I=\left\|\int_\eta^tf_1(s)f_2(t-s)\,\rd s\right\|_{L^1(\eta,T_0+\de)}\]
and thus \eqref{eq-RCI-1} is exactly \eqref{eq-RCI-a}.\medskip

(b) Considering the possibility of $f_1\le0$ on $(T_0,T_0+\de)$, we split $I$ into three parts and continue to estimate
\begin{align*}
I & =\int_\eta^{T_0}\!\!\!\!\int_\eta^tf_1(s)f_2(t-s)\,\rd s\rd t+\int_{T_0}^{T_0+\de}\!\!\!\!\int_\eta^{T_0}f_1(s)f_2(t-s)\,\rd s\rd t+\int_{T_0}^{T_0+\de}\!\!\!\!\int_{T_0}^t|f_1(s)|\,f_2(t-s)\,\rd s\rd t\\
& \le\int_\eta^{T_0}\!\!\!\!\int_\eta^tf_1(s)f_2(t-s)\,\rd s\rd t+\int_{T_0}^{T_0+\de}\!\!\!\!\int_\eta^{T_0}f_1(s)f_2(t-s)\,\rd s\rd t+\int_{T_0}^{T_0+\de}\!\!\!\!\int_{T_0}^tf_1(s)f_2(t-s)\,\rd s\rd t\\
& \quad\,+2\int_{T_0}^{T_0+\de}\!\!\!\!\int_{T_0}^t|f_1(s)|\,f_2(t-s)\,\rd s\rd t\\
& =\int_\eta^{T_0+\de}\!\!\!\!\int_\eta^tf_1(s)f_2(t-s)\,\rd s\rd t+2\int_{T_0}^{T_0+\de}\!\!\!\!\int_{T_0}^t|f_1(s)|\,f_2(t-s)\,\rd s\rd t\\
& \le\left\|\int_\eta^tf_1(s)f_2(t-s)\,\rd s\right\|_{L^1(\eta,T_0+\de)}+2\int_{T_0}^{T_0+\de}\!\!\!\!\int_{T_0}^t|f_1(s)|\,f_2(t-s)\,\rd s\rd t.
\end{align*}
Since the term $\int_{T_0}^t|f_1(s)|\,f_2(t-s)\,\rd s$ is the convolution of $|f_1|$ and $f_2$ starting from $t=T_0$, it follows immediately from Young's inequality that
\[\int_{T_0}^{T_0+\de}\!\!\!\!\int_{T_0}^t|f_1(s)|\,f_2(t-s)\,\rd s\rd t\le\|f_1\|_{L^1(T_0,T_0+\de)}\|f_2\|_{L^1(0,\de)},\]
and consequently
\[I\le\left\|\int_\eta^tf_1(s)f_2(t-s)\,\rd s\right\|_{L^1(\eta,T_0+\de)}+2\|f_1\|_{L^1(T_0,T_0+\de)}\|f_2\|_{L^1(0,\de)}.\]
The combination of the above inequality with \eqref{eq-RCI-1} finishes the proof of \eqref{eq-RCI-b}.
\end{proof}

\begin{proof}[Proof of Theorem $\ref{thm-stab}$]
Fix $\rho\in\cU_N$ arbitrarily. As was explained before, it suffices to give an estimate for $\rho$ in the convolution \eqref{eq-Duhamel-J} with $x=x_0$.

First we mention that both $v(x_0,\,\cdot\,)$ and $u_\rho(x_0,\,\cdot\,)$ make sense in $L^1(0,T)$. Actually, for the whole space case \eqref{eq-ivp-v}, $v$ is pointwisely defined according to Lemma \ref{lem-ivp}. For the bounded domain case \eqref{eq-ibvp-v}, it follows from \cite[Theorem 2.1]{LLY15} that $v\in L^1(0,T;H^2(\Om)\cap H_0^1(\Om))$. Since we restrict $d\le3$, the Sobolev embedding theorem implies $v(x_0,\,\cdot\,)\in L^1(0,T)$. On the other hand, we incorporate Lemma \ref{lem-Caputo-RL} with $\|\rho\|_{C^1[0,T]}\le M$ to estimate for $0<t\le T$ that
\begin{align*}
|D_t^{1-\al}\rho(t)| & =\left|\f{\rho(0)}{\Ga(\al)}t^{\al-1}+\pa_t^{1-\al}\rho(t)\right|\le\f{|\rho(0)|}{\Ga(\al)}t^{1-\al}+\f1{\Ga(\al)}\int_0^t\f{|\rho'(s)|}{(t-s)^{1-\al}}\,\rd s\\
& \le\f M{\Ga(\al)}\left(t^{\al-1}+\f{t^\al}{\al}\right),
\end{align*}
implying $D_t^{1-\al}\rho\in L^1(0,T)$. Hence, Young's inequality immediately gives $u_\rho(x_0,\,\cdot\,)\in L^1(0,T)$.

Next, since the initial value $g$ of $v$ is assumed to be non-negative and non-vanishing, we know $v>0$ in $\BR^d\times(0,\infty)$ in the case of \eqref{eq-ivp-v} by Lemma \ref{lem-ivp}, and $v>0$ a.e.\! in $\Om\times(0,\infty)$ in the case of \eqref{eq-ibvp-v} by the strong maximum principle established in \cite[Theorem 1.2]{LRY16}. Now that $v(x_0,\,\cdot\,)>0$ a.e.\! in $(0,T)$, it is readily seen that the constant $B_\de$ defined in \eqref{eq-def-Bd} is strictly positive and well-defined for any $\de>0$.

Now we are in a position to substitute $f_1:=\rho$ and $f_2:=v(x_0,\,\cdot\,)$ into Lemma \ref{lem-rci} and discuss the case of $N=0$ and $N>0$ respectively.\medskip

(a) As $\rho\in\cU_0$ means that $\rho$ keeps the sign on $(0,T)$, we simply take $\eta=0$ and $T_0=T-\de$ ($0<\de<T$) in \eqref{eq-RCI-a} to obtain
\begin{align*}
\|\rho\|_{L^1(0,T-\de)} & \le\f1{\|v(x_0,\,\cdot\,)\|_{L^1(0,\de)}}\left\|\int_0^t\rho(s)\,v(x_0,t-s)\,\rd s\right\|_{L^1(0,T)}\\
& =B_\de\|J^{1-\al}u_\rho(x_0,\,\cdot\,)\|_{L^1(0,T)}\le\f{T^{1-\al}B_\de}{\Ga(2-\al)}\|u_\rho(x_0,\,\cdot\,)\|_{L^1(0,T)},
\end{align*}
where we applied \eqref{eq-est-RL} in the last inequality.\medskip

(b) Now let $N=1,2,\ldots$ and $\de\in(0,1)$ be sufficiently small. Since $\rho\in\cU_N$, by definition we can assume that $\rho$ changes signs at $0<t_1<t_2<\cdots<t_{N'}<T-\de$ ($N'\le N$) on $(0,T-\de)$. For later convenience, additionally set $t_0:=0$ and $t_{N'+1}:=T-\de$. Taking $x=x_0$ in \eqref{eq-Duhamel-J} and rearranging according to the sign of $\rho$, we can write
\[\int_{t_k}^t\rho(s)\,v(x_0,t-s)\,\rd s=J^{1-\al}u_\rho(x_0,t)-\sum_{j=1}^kI_j,\quad k=0,1,\ldots,N',\]
where
\[I_j:=\int_{t_{j-1}}^{t_j}\rho(s)\,v(x_0,t-s)\,\rd s,\quad j=1,\ldots,k.\]
Applying \eqref{eq-RCI-b} in Lemma \ref{lem-rci} to the above equality with $\eta=t_k$ and $T_0=t_{k+1}$, we obtain
\begin{align}
\|\rho\|_{L^1(t_k,t_{k+1})} & \le B_\de\left\|\int_{t_k}^t\rho(s)\,v(x_0,t-s)\,\rd s\right\|_{L^1(t_k,t_{k+1}+\de)}+2\|\rho\|_{L^1(t_{k+1},t_{k+1}+\de)}\nonumber\\
& \le B_\de\left(\|J^{1-\al}u_\rho(x_0,\,\cdot\,)\|_{L^1(t_k,t_{k+1}+\de)}+\sum_{j=1}^k\|I_j\|_{L^1(t_k,t_{k+1}+\de)}\right)\nonumber\\
& \quad\,+2\|\rho\|_{L^1(t_{k+1},t_{k+1}+\de)},\quad k=0,1,\ldots,N'.\label{eq-est-0}
\end{align}
For $\|\rho\|_{L^1(t_{k+1},t_{k+1}+\de)}$, it follows from $\|\rho\|_{C^1[0,T]}\le M$ and the fact that $\rho(t_{k+1})=0$ that
\begin{equation}\label{eq-est-1}
\|\rho\|_{L^1(t_{k+1},t_{k+1}+\de)}\le\int_{t_{k+1}}^{t_{k+1}+\de}\!\!\!\int_{t_{k+1}}^{t_{k+1}+t}|\rho'(s)|\,\rd s\rd t\le\f M2\de^2.
\end{equation}
For $\|I_j\|_{L^1(t_k,t_{k+1}+\de)}$, we change the order of integrals to deduce
\begin{align}
\|I_j\|_{L^1(t_k,t_{k+1}+\de)} & =\int_{t_k}^{t_{k+1}+\de}\!\!\!\int_{t_{j-1}}^{t_j}|\rho(s)|\,v(x_0,t-s)\,\rd s\rd t=\int_{t_{j-1}}^{t_j}|\rho(s)|\int_{t_k}^{t_{k+1}+\de}v(x_0,t-s)\,\rd t\rd s\nonumber\\
& =\int_{t_{j-1}}^{t_j}|\rho(s)|\|v(x_0,\,\cdot\,)\|_{L^1(t_k-s,t_{k+1}+\de-s)}\,\rd s.\label{eq-est-2}
\end{align}
To deal with $\|v(x_0,\,\cdot\,)\|_{L^1(t_k-s,t_{k+1}+\de-s)}$, we shall turn to the explicit solution to give a pointwise estimate of $|v(x_0,t)|$ for $0<t<T$. To this end, we discuss the whole space case \eqref{eq-ivp-v} and the bounded domain case \eqref{eq-ibvp-v} separately.

For the Cauchy problem \eqref{eq-ivp-v}, we recall the formula \eqref{eq-ivp-fdm} in Lemma \ref{lem-ivp} and the assumption \eqref{eq-asp-g1} to bound
\begin{align}
|v(x_0,t)| & \le\int_{\BR^d}|K_\al(x-x_0,t)||g(x)|\,\rd x\nonumber\\
& \le\|g\|_{C(\BR^d)}\left(\int_{|x-x_0|\le t^{\al/2}}|K_\al(x-x_0,t)|\,\rd x+\int_{|x-x_0|>t^{\al/2}}|K_\al(x-x_0,t)|\,\rd x\right)\nonumber\\
& \le\|g\|_{C(\BR^d)}\left(J_1+c_1\,t^{-\f{\al d}2}J_2\right),\label{eq-est-2'}
\end{align}
where we apply \eqref{eq-Fox-est1} to treat the case of $|x-x_0|>t^{\f\al2}$ and define
\[J_1:=\int_{|x-x_0|\le t^{\al/2}}|K_\al(x-x_0,t)|\,\rd x,\quad J_2:=\int_{|x-x_0|>t^{\al/2}}\exp\left(-c_2\,t^{-\f\al{2-\al}}|x-x_0|^{\f2{2-\al}}\right)\rd x.\]
We first use \eqref{eq-Fox-est0} to treat $J_1$. For $d=1$, we immediately get $J_1\le2c_1$. For $d=2$, we change to the polar coordinate to estimate
\[J_1\le c_1\,t^{-\al}\int_{|x-x_0|\le t^{\al/2}}\left(\left|\log\left(\f{|x-x_0|^2}{t^\al}\right)\right|+1\right)\rd x=2\pi c_1\int_0^1r\,(1-2\log r)\,\rd r=2\pi c_1.\]
Similarly, for $d=3$ we have
\[J_1\le c_1\,t^{-\al}\int_{|x-x_0|\le t^{\al/2}}\f{\rd x}{|x-x_0|}=4\pi c_1\,t^{-\al}\int_0^{t^{\al/2}}r\,\rd r=2\pi c_1.\]
By a parallel argument, we can estimate $J_2$ as
\[J_2\le V_d\f{d(2-\al)}2\int_{c_2}^\infty s^{\f{d(2-\al)}2-1}\e^{-s}\,\rd s\le V_d\,\Ga\left(1+d\left(1-\f\al2\right)\right)\left(c_2^{-\f{2-\al}2}t^{\f\al2}\right)^d,\]
where $V_d$ denotes the volume of a $d$-dimensional unit ball, and we used the definition of Gamma functions. Substituting all the above estimates into \eqref{eq-est-2'}, we see that $|v(x_0,t)|$ is dominated by a single constant $c_1'>0$ for all $t>0$, and hence
\begin{equation}\label{eq-est-2a}
\|v(x_0,\,\cdot\,)\|_{L^1(t_k-s,t_{k+1}+\de-s)}\le c_1'(t_{k+1}+\de-t_k)\le c_1'\,T=:C_0'.
\end{equation}

For the initial-boundary value problem \eqref{eq-ibvp-v}, we invoke the explicit solution by the eigensystem expansion in Lemma \ref{lem-ibvp}(a). Since we assume $g\in\cD((-\tri)^\ve)$ with some $\ve>0$ in \eqref{eq-asp-g2}, we argue similarly as that in \cite{LLY15} to estimate
\begin{align*}
\|v(\,\cdot\,,t)\|_{\cD((-\tri)^{d/4+\ve})}^2 & =\sum_{n=1}^\infty\left|\la_n^{\f d4}E_{\al,1}(-\la_nt^\al)\right|^2\left|\la_n^\ve(g,\vp_n)\right|^2\\
& \le\left(c_0\,t^{-\f{\al d}4}\right)^2\sum_{n=1}^\infty\left|\f{(\la_nt^\al)^{\f d4}}{1+\la_nt^\al}\right|^2|\la_n^\ve(g,\vp_n)|^2\le\left(c_0\|g\|_{\cD((-\tri)^\ve)}t^{-\f{\al d}4}\right)^2,
\end{align*}
where the Mittag-Leffler function $E_{\al,1}(-\la_nt^\al)$ is treated by \eqref{eq-est-ML}. Then it follows from the embedding $\cD((-\tri)^{\f d4+\ve})\subset H^{\f d2+2\ve}(\Om)\subset\subset C(\ov\Om)$ that
\begin{align*}
|v(x_0,t)| & \le\|v(\,\cdot\,,t)\|_{C(\ov\Om)}\le c_0'\|v(\,\cdot\,,t)\|_{\cD((-\tri)^{d/4+\ve})}\le c_0'c_0\|g\|_{\cD((-\tri)^\ve)}t^{-\f{\al d}4},
\end{align*}
implying
\begin{align}
\|v(x_0,\,\cdot\,)\|_{L^1(t_k-s,t_{k+1}+\de-s)} & \le c_0'c_0\|g\|_{\cD((-\tri)^\ve)}\int_{t_k-s}^{t_{k+1}+\de-s}t^{-\f{\al d}4}\,\rd t\nonumber\\
& \le\f{4c_0'c_0\|g\|_{\cD((-\tri)^\ve)}}{4-\al d}T^{1-\f{\al d}4}=:C_0''.\label{eq-est-2b}
\end{align}
Plugging \eqref{eq-est-2a} and \eqref{eq-est-2b} into \eqref{eq-est-2}, we conclude
\[\|I_j\|_{L^1(t_k,t_{k+1}+\de)}\le C_0\|\rho\|_{L^1(t_{j-1},t_j)},\quad C_0:=\max(C_0',C_0''),\ j=1,\ldots,k.\]
Consequently, we substitute the above inequality and \eqref{eq-est-1} into \eqref{eq-est-0} to deduce
\begin{equation}\label{eq-est-3}
\|\rho\|_{L^1(t_k,t_{k+1})}\le B_\de\left(\|J^{1-\al}u_\rho(x_0,\,\cdot\,)\|_{L^1(t_k,t_{k+1}+\de)}+C_0\|\rho\|_{L^1(0,t_k)}\right)+M\de^2
\end{equation}
for $k=0,1,\ldots,N'$. Especially, taking $k=0$ in \eqref{eq-est-3} yields
\begin{equation}\label{eq-est-3'}
\|\rho\|_{L^1(0,t_1)}\le B_\de\|J^{1-\al}u_\rho(x_0,\,\cdot\,)\|_{L^1(0,t_1+\de)}+M\de^2.
\end{equation}

Now we employ an inductive argument to show for $k=1,2,\ldots,N'+1$ that
\begin{equation}\label{eq-est-4}
\|\rho\|_{L^1(0,t_k)}\le\f{(C_0B_\de+1)^k-1}{C_0B_\de}\left(B_\de\|J^{1-\al}u_\rho(x_0,\,\cdot\,)\|_{L^1(0,t_k+\de)}+M\de^2\right).
\end{equation}
In fact, \eqref{eq-est-4} with $k=1$ is exactly \eqref{eq-est-3'}. Supposing that \eqref{eq-est-4} is valid for some $k\le N'$, we can apply \eqref{eq-est-3} to bound $\|\rho\|_{L^1(0,t_{k+1})}$ as
\begin{align*}
\|\rho\|_{L^1(0,t_{k+1})} & \le\|\rho\|_{L^1(0,t_k)}+B_\de\left(\|J^{1-\al}u_\rho(x_0,\,\cdot\,)\|_{L^1(t_k,t_{k+1}+\de)}+C_0\|\rho\|_{L^1(0,t_k)}\right)+M\de^2\\
& \le(C_0B_\de+1)\f{(C_0B_\de+1)^k-1}{C_0B_\de}\left(B_\de\|J^{1-\al}u_\rho(x_0,\,\cdot\,)\|_{L^1(0,t_k+\de)}+M\de^2\right)\\
& \quad\,+\left(B_\de\|J^{1-\al}u_\rho(x_0,\,\cdot\,)\|_{L^1(t_k,t_{k+1}+\de)}+M\de^2\right)\\
& \le\f{(C_0B_\de+1)^{k+1}-1}{C_0B_\de}\left(B_\de\|J^{1-\al}u_\rho(x_0,\,\cdot\,)\|_{L^1(0,t_{k+1}+\de)}+M\de^2\right),
\end{align*}
which is indeed \eqref{eq-est-4} with $k+1$. Taking $k=N'+1$ in \eqref{eq-est-4} and recalling $t_{N'+1}=T-\de$, we finally arrive at the desired inequality \eqref{eq-stab-b} with the aid of \eqref{eq-est-RL}.
\end{proof}

Next, under the smallness assumption of the observation data, we proceed to the proof of Theorem \ref{thm-stab-log} on the basis of key estimates \eqref{eq-est-3} and \eqref{eq-est-3'}. The main idea is to choose different parameters $\de$ to minimize the estimate on each interval where $D_t^{1-\al}\rho$ keeps the sign.

\begin{proof}[Proof of Theorem $\ref{thm-stab-log}$]
To further characterize the stability described in \eqref{eq-stab-b} of Theorem \ref{thm-stab}(b), it is necessary to give an estimate for the positive constant $B_\de$ defined in \eqref{eq-def-Bd}. Obviously, this equals to giving a lower bound of $\|v(x_0,\,\cdot\,)\|_{L^1(0,\de)}$, where $v$ solves the homogeneous problem \eqref{eq-ivp-v} or \eqref{eq-ibvp-v}. In the sequel, we always assume that $\de>0$ is sufficiently small.

On the other hand, as before we assume that $D_t^{1-\al}\rho$ changes signs at $0<t_1<t_2<\cdots<t_{N'}<T$ ($N'\le N$) on $(0,T)$ and additionally set $t_0:=0$ and $t_{N'+1}:=T-\de_*$. Here $\de_*>0$ is a sufficiently small constant to be determined later.

Write $D:=\|u_\rho(x_0,\,\cdot\,)\|_{L^1(0,T)}$ for simplicity, and recall that $D$ was assumed to be sufficiently small. Since the uniqueness (i.e., the case of $D=0$) for Problem \ref{prob-isp} was proved in \cite{LRY16}, henceforth we may assume $D>0$ without loss of generality.\medskip

(a) As $x_0\in\supp\,g$ and $g\ge0$ is assumed to be continuous in $\BR^d$ or $\ov\Om\,$, we know $g(x_0)>0$. Due to the continuity of the solution, there holds $v(x_0,t)\ge\f{g(x_0)}2$ for sufficiently small $t>0$ and thus
\begin{equation}\label{eq-est-Bd0}
B_\de\le\f2{g(x_0)}\f1\de.
\end{equation}
On the first interval $(0,t_1)$, we substitute \eqref{eq-est-Bd0} into \eqref{eq-est-3'} and use \eqref{eq-est-RL} to obtain
\[\|\rho\|_{L^1(0,t_1)}\le\f{T^{1-\al}B_\de}{\Ga(2-\al)}\|u_\rho(x_0,\,\cdot\,)\|_{L^1(0,t_1+\de)}+M\de^2\le\f{2\,T^{1-\al}}{\Ga(2-\al)\,g(x_0)}\f D
\de+M\de^2\le a_0\left(\f D\de+\de^\al\right),\]
where $a_0:=\max(\f{2\,T^{1-\al}}{\Ga(2-\al)\,g(x_0)},M)$. Since $D$ is sufficiently small, we balance the right-hand side of the above inequality by choosing $\de_1:=D^{\f13}$, which gives
\begin{equation}\label{eq-est-6h}
\|\rho\|_{L^1(0,t_1)}\le2a_0\,\de_1^2=2a_0\,D^{\f23}.
\end{equation}
Although the choice of $\de_1$ here differs from the minimizer up to a multiplier, it gives the same asymptotic behavior with respect to $D$, and henceforth we will adopt such a balancing argument for simplicity.

Now we utilize an inductive argument to show for $k=0,1,\ldots,N'$ that
\begin{equation}\label{eq-est-7h}
\|\rho\|_{L^1(t_k,t_{k+1})}\le2a_k\,D^{(\f23)^{k+1}},\quad\|\rho\|_{L^1(0,t_{k+1})}\le2\sum_{j=0}^ka_j\,D^{(\f23)^{k+1}},
\end{equation}
where $a_k$ is defined inductively by
\begin{equation}\label{eq-def-ak}
a_k:=\max\left\{\f2{g(x_0)}\left(\f{T^{1-\al}}{\Ga(2-\al)}+2C_0\sum_{j=0}^{k-1}a_j\right),M\right\}.
\end{equation}
In fact, \eqref{eq-est-6h} is exactly the case of $k=0$. Supposing that \eqref{eq-est-7h} holds for some $k-1$, we apply \eqref{eq-est-RL} and substitute \eqref{eq-est-7h} with $k-1$ and \eqref{eq-est-Bd0} into \eqref{eq-est-3} to treat the interval $(t_k,t_{k+1})$ as
\begin{align*}
\|\rho\|_{L^1(t_k,t_{k+1})} & \le\f2{g(x_0)}\f1\de\left(\f{T^{1-\al}}{\Ga(2-\al)}\|u_\rho(x_0,\,\cdot\,)\|_{L^1(t_k,t_{k+1}+\de)}+2C_0\sum_{j=0}^{k-1}a_j\,D^{(\f23)^k}\right)+M\de^2\\
& \le\f2{g(x_0)}\left(\f{T^{1-\al}}{\Ga(2-\al)}+2C_0\sum_{j=0}^{k-1}a_j\right)\f1\de D^{(\f23)^k}+M\de^2\le a_k\left(\f1\de D^{(\f23)^k}+\de^2\right),
\end{align*}
where we used the smallness of $D$ to bound $\|u_\rho(x_0,\,\cdot\,)\|_{L^1(t_k,t_{k+1}+\de)}\le D\le D^{(\f23)^k}$, and $a_k$ takes the form of \eqref{eq-def-ak}. Here we choose
\begin{equation}\label{eq-def-de0}
\de_{k+1}:=D^{\f13(\f23)^k}
\end{equation}
to balance the right-hand side of the above inequality, so that
\begin{align*}
\|\rho\|_{L^1(t_k,t_{k+1})} & \le2a_k\,\de_{k+1}^2=2a_k\,D^{(\f23)^{k+1}},\\
\|\rho\|_{L^1(0,t_{k+1})} & \le\|\rho\|_{L^1(0,t_k)}+\|\rho\|_{L^1(t_k,t_{k+1})}\le2\sum_{j=0}^{k-1}a_j\,D^{(\f23)^k}+2a_k\,D^{(\f23)^{k+1}}\le2\sum_{j=0}^ka_j\,D^{(\f23)^{k+1}},
\end{align*}
as claimed in \eqref{eq-est-7h}.

Finally, we take $k=N'$ in \eqref{eq-est-7h} and recall $t_{N'+1}=T-\de_*$ to conclude
\[\|\rho\|_{L^1(0,T-\de_*)}\le C_1\,D^{(\f23)^{N'+1}}\le C_1\,D^{(\f23)^{N+1}},\quad C_1:=2\sum_{j=0}^{N'}a_j.\]
According to \eqref{eq-def-de0}, this is achieved by choosing
\[\de_*=\de_{N'+1}=D^{\f13(\f23)^{N'}},\]
and it follows from the smallness of $D$ that $\de_*$ is also sufficiently small. However, since the sign changes of $\rho$ is independent of $\de_*$, it is possible that the interval $(T-\de_*,T)$ includes several of the points $t_k$ ($1\le k\le N'$). In this case, the induction for \eqref{eq-est-7h} automatically stops before taking $k=N'$, indicating that the above estimate is indeed the worst case.\medskip

(b) If we do not impose $x_0\in\supp\,g$, as before we shall discuss the whole space case \eqref{eq-ivp-v} and the bounded domain case \eqref{eq-ibvp-v} separately to estimate $B_\de$.

First we treat the Cauchy problem \eqref{eq-ivp-v}. Since we assume the compactness of $\supp\,g$, it is readily seen that $\ka_0=\sup_{x\in\supp\,g}|x-x_0|>0$ is finite. Meanwhile, due to $x_0\not\in\ov{\supp\,g}$, for sufficiently small $t>0$ we can substitute the asymptotic behavior \eqref{eq-Fox-asp} into \eqref{eq-ivp-fdm} to deduce
\begin{align}
v(x_0,t) & =\int_{\BR^d}K_\al(x-x_0,t)\,g(x)\,\rd x\nonumber\\
& \ge\f{c_3}2\,t^{-\f{\al d}{2(2-\al)}}\int_{\BR^d}|x-x_0|^{-\f{d(1-\al)}{2-\al}}\exp\left(-c_4\,t^{-\f\al{2-\al}}|x-x_0|^{\f2{2-\al}}\right)g(x)\,\rd x\nonumber\\
& \ge c_5\,t^{-\f{\al d}{2(2-\al)}}\exp\left(-c_6\,t^{-\f\al{2-\al}}\right),\label{eq-est-5}
\end{align}
where
\[c_5:=\f{c_3}2\ka_0^{-\f{d(1-\al)}{2-\al}}\|g\|_{L^1(\BR^d)},\quad c_6:=c_4\,\ka_0^{\f2{2-\al}}.\]
Then we calculate
\begin{align*}
\|v(x_0,\,\cdot\,)\|_{L^1(0,\de)} & \ge c_5\int_0^\de t^{-\f{\al d}{2(2-\al)}}\exp\left(-c_6\,t^{-\f\al{2-\al}}\right)\rd t\\
& =\left(\f2\al-1\right)c_5\,c_6^{\f2\al-1-\f d2}\int_{c_6\de^{-\al/(2-\al)}}^\infty\tau^{-(\f2\al-\f d2)}\e^{-\tau}\,\rd\tau.
\end{align*}
To further estimate the above integral from below, we invoke L'Hospital's rule to see
\[\int_R^\infty\tau^{-\be}\e^{-\tau}\,\rd\tau=R^{-\be}\e^{-R}(1+o(1))\ge\f12R^{-\be}\e^{-R}\]
for $\be>0$ and sufficiently large $R>0$. This implies
\[\int_{c_6\de^{-\al/(2-\al)}}^\infty\tau^{-(\f2\al-\f d2)}\e^{-\tau}\,\rd\tau\ge\f12c_6^{-(\f2\al-\f d2)}\de^{\f{4-\al d}{2(2-\al)}}\exp\left(-c_4\,\de^{-\f\al{2-\al}}\right),\]
which eventually leads us to the lower estimate
\begin{equation}\label{eq-est-de}
\|v(x_0,\,\cdot\,)\|_{L^1(0,\de)}\ge\f{(2-\al)c_5}{2\al\,c_6}\,\de^{\f{4-\al d}{2(2-\al)}}\exp\left(-c_6\,\de^{-\f\al{2-\al}}\right).
\end{equation}

Next we deal with the initial-boundary value problem \eqref{eq-ibvp-v} with $\al=1$ and the assumption \eqref{eq-asp-x0}. Taking $x=x_0$ in the representation \eqref{eq-ibvp-fdm} in Lemma \ref{lem-ibvp}(b), we have
\[v(x_0,t)=v_1(t)+v_2(t),\]
where
\begin{align*}
v_1(t) & :=(4\pi\,t)^{-\f d2}\int_\Om\exp\left(-\f{|x-x_0|^2}{4\,t}\right)g(x)\,\rd x,\\
v_2(t) & :=\int_0^t(4\pi(t-s))^{-\f d2}\int_{\pa\Om}\exp\left(-\f{|y-x_0|^2}{4(t-s)}\right)\pa_\nu v(y,s)\,\rd y\rd s.
\end{align*}
For $v_1(t)$, the same argument as that for \eqref{eq-est-5} yields for sufficiently small $t>0$ that
\begin{equation}\label{eq-est-5'}
v_1(t)\ge c_5'\,t^{-\f d2}\exp(-\ka_0'\,t^{-1}),\quad c_5':=(4\pi)^{-\f d2}\|g\|_{L^1(\Om)},\ \ka_0':=\f{\ka_0^2}4.
\end{equation}
For $v_2(t)$, the Hopf lemma asserts $\pa_\nu v<0$ on $\pa\Om\times(0,T)$ and thus $v_2(t)<0$ by definition. Then it is necessary to give an upper bound for $\pa_\nu v$ in order to estimate $|v_2(t)|$. To this end, we turn to the explicit solution \eqref{eq-ibvp-sol} again and the trace theorem to dominate
\begin{align*}
\|\pa_\nu v(\,\cdot\,,t)\|_{L^1(\pa\Om)} & \le\|\nb v(\,\cdot\,,t)\|_{L^1(\pa\Om)}\le c_\Om\|\nb v(\,\cdot\,,t)\|_{L^2(\pa\Om)}\le c_\Om'\|v(\,\cdot\,,t)\|_{H^{3/2}(\Om)}\\
& =c_\Om'\left(\sum_{n=1}^\infty\left|\la_n^{\f34}\e^{-\la_nt}(g,\vp_n)\right|^2\right)^{\f12}\le c_\Om'\left(\f3{4\,\e\,t}\right)^{\f34}\|g\|_{L^2(\Om)}.
\end{align*}
Therefore, we obtain
\begin{align*}
|v_2(t)| & \le(4\pi)^{-\f d2}\int_0^t(t-s)^{-\f d2}\exp\left(-\f{\inf_{y\in\pa\Om}|y-x_0|^2}{4(t-s)}\right)\|\pa_\nu v(\,\cdot\,,s)\|_{L^1(\pa\Om)}\,\rd s\\
& \le\f{c_\Om'\|g\|_{L^2(\Om)}}{(4\pi)^{\f d2}}\left(\f3{4\,\e}\right)^{\f34}\int_0^t(t-s)^{-\f d2}\exp\left(-\f{\ka_1^2}{4(t-s)}\right)s^{-\f34}\,\rd s\\
& \le\f{4c_\Om'\|g\|_{L^2(\Om)}}{(4\pi)^{\f d2}}\left(\f3{4\,\e}\right)^{\f34}t^{\f14}\int_0^ts^{-\f d2}\exp(-\ka_1'\,s^{-1})\rd s,
\end{align*}
where $\ka_1':=\f14\ka_1^2$ and we applied Young's inequality to the involved convolution. For sufficiently small $t>0$, we use again L'Hospital's rule to obtain
\[\int_0^ts^{-\f d2}\exp\left(-\f{\ka_1'}s\right)\rd s=\f1{\ka_1'}\,t^{2-\f d2}\exp(-\ka_1'\,t^{-1})(1+o(1))\le\f2{\ka_1'}\,t^{2-\f d2}\exp(-\ka_1'\,t^{-1}),\]
indicating
\[|v_2(t)|\le c_6'\,t^{\f{9-2d}4}\exp(-\ka_1'\,t^{-1}),\quad c_6':=\f{8c_\Om'\|g\|_{L^2(\Om)}}{(4\pi)^{\f d2}\ka_1'}\left(\f3{4\,\e}\right)^{\f34}.\]
Combining the above inequality with \eqref{eq-est-5'} yields
\[v(x_0,t)\ge c_5'\,t^{-\f d2}\exp(-\ka_0'\,t^{-1})-c_6'\,t^{\f{9-2d}4}\exp(-\ka_1'\,t^{-1})\]
for sufficiently small $t>0$. Thanks to the assumption \eqref{eq-asp-x0}, there holds $\ka_0'\le\ka_1'$, implying that the term $t^{-\f d2}\exp(-\ka_0'\,t^{-1})$ is dominating and thus
\[v(x_0,t)\ge\f{c_5'}2\,t^{-\f d2}\exp(-\ka_0'\,t^{-1})\]
for sufficiently small $t>0$. Repeating the same argument using L'Hospital rule, we conclude
\begin{align*}
\|v(x_0,\,\cdot\,)\|_{L^1(0,\de)} & \ge\f{c_5'}2\int_0^\de t^{-\f d2}\exp(-\ka_0'\,t^{-1})\,\rd t=\f{c_5'}{2\ka_0'}\,\de^{2-\f d2}\exp(-\ka_0'\,\de^{-1})(1+o(1))\\
& \ge\f{c_5'}{4\ka_0'}\,\de^{2-\f d2}\exp(-\ka_0'\,\de^{-1}).
\end{align*}
Noting that the above inequality coincides well with \eqref{eq-est-de} for the whole space case, finally we obtain a uniform lower bound for $\|v(x_0,\,\cdot\,)\|_{L^1(0,\de)}$ in both cases, which implies the following estimate for $B_\de$:
\[B_\de\le C_1'\,\de^{\f{\al d-4}{2(2-\al)}}\exp(\ka'\,\de^{-\f\al{2-\al}}),\quad C_1':=\max\left(\f{(2-\al)c_5}{2\al\,c_6},\f{4\ka_0'}{c_5'}\right),\ \ka':=\max(c_6,\ka_0').\]
Since the growth rate of the term $\exp(\ka'\,\de^{-\f\al{2-\al}})$ is much faster than any negative power of $\de$ as $\de\downarrow0$, there exists a constant $\ka>\ka'$ such that for sufficiently small $\de>0$, there holds
\begin{equation}\label{eq-est-Bd}
B_\de\le C_1'\,\de^2\exp(\ka\,\de^{-\f\al{2-\al}}).
\end{equation}

Now we can follow the same line as that for the case of $x_0\in\supp\,g$ to prove \eqref{eq-stab-d}. We start from the first interval $(0,t_1)$ by using \eqref{eq-est-3'} and \eqref{eq-est-Bd} to estimate
\begin{align*}
\|\rho\|_{L^1(0,t_1)} & \le\f{C_1'\,T^{1-\al}}{\Ga(2-\al)}D\,\de^2\exp(\ka\,\de^{-\f\al{2-\al}})+M\de^2\\
& \le b_0\left(\f\al{2(2-\al)}\right)^{2(\f2\al-1)}\de^2\left(D\exp(\ka\,\de^{-\f\al{2-\al}})+1\right),
\end{align*}
where $b_0:=(2(\f2\al-1))^{2(\f2\al-1)}\max(\f{C_1'\,T^{1-\al}}{\Ga(2-\al)},M)$. Since $D$ is sufficiently small, we can choose
\[\de_1:=\left(\f\ka{|\log D|}\right)^{\f2\al-1}.\]
to balance the right-hand side of the above inequality. Then $\de_1$ is also sufficiently small, and consequently
\begin{equation}\label{eq-est-6}
\|\rho\|_{L^1(0,t_1)}\le 2b_0\left(\f\al{2(2-\al)}\right)^{2(\f2\al-1)}\de_1^2=2b_0\left(\f{\ka\,\al}{2(2-\al)}\right)^{2(\f2\al-1)}|\log D|^{-2(\f2\al-1)}.
\end{equation}

Now we shall take advantage of the induction to show for $k=0,1,\ldots,N'$ that
\begin{align}
& \|\rho\|_{L^1(t_k,t_{k+1})}\le2b_k\left(\f{\ka\,\al}{2(2-\al)}\right)^{2(\f2\al-1)}(\underbrace{\log(\cdots(\log|\log}_{k+1}D|)\cdots))^{-2(\f2\al-1)},\label{eq-est-7}\\
& \|\rho\|_{L^1(0,t_{k+1})}\le2\left(\f{\ka\,\al}{2(2-\al)}\right)^{2(\f2\al-1)}\sum_{j=0}^kb_j\,(\underbrace{\log(\cdots(\log|\log}_{k+1}D|)\cdots))^{-2(\f2\al-1)},\label{eq-est-7'}
\end{align}
where $b_k$ is defined inductively by
\begin{equation}\label{eq-def-bk}
b_k:=\max\left\{C_1'\left(\f{T^{1-\al}}{\Ga(2-\al)}+2C_0\left(\f{\ka\,\al}{2(2-\al)}\right)^{2(\f2\al-1)}\sum_{j=0}^{k-1}b_j\right),M\right\}.
\end{equation}
In fact, \eqref{eq-est-6} is exactly the case of $k=0$. Now suppose that \eqref{eq-est-7} and \eqref{eq-est-7'} hold for some $k-1$, and we are in a position to proceed to the case of $k$. Applying \eqref{eq-est-3}, \eqref{eq-est-Bd} and \eqref{eq-est-RL}, we obtain
\begin{align*}
\|\rho\|_{L^1(t_k,t_{k+1})} & \le C_1'\,\de^2\exp(\ka\,\de^{-\f\al{2-\al}})\left(\f{T^{1-\al}}{\Ga(2-\al)}\|u_\rho(x_0,\,\cdot\,)\|_{L^2(t_k,t_{k+1}+\de)}+C_0\|\rho\|_{L^1(0,t_k)}\right)\\
& \quad\,+M\de^2.
\end{align*}
Then we utilize \eqref{eq-est-7'} with $k-1$ and the smallness of $D$ to further bound
\begin{align*}
& \quad\,\f{T^{1-\al}}{\Ga(2-\al)}\|u_\rho(x_0,\,\cdot\,)\|_{L^1(t_k,t_{k+1}+\de)}+C_0\|\rho\|_{L^1(0,t_k)}\\
& \le\f{T^{1-\al}}{\Ga(2-\al)}D+2C_0\left(\f{\ka\,\al}{2(2-\al)}\right)^{2(\f2\al-1)}\sum_{j=0}^{k-1}b_j\,(\underbrace{\log(\cdots|\log}_kD|\cdots))^{-2(\f2\al-1)}\\
& \le\left(\f{T^{1-\al}}{\Ga(2-\al)}+2C_0\left(\f{\ka\,\al}{2(2-\al)}\right)^{2(\f2\al-1)}\sum_{j=0}^{k-1}b_j\right)(\underbrace{\log(\cdots|\log}_kD|\cdots))^{-2(\f2\al-1)}.
\end{align*}
This indicates
\[\|\rho\|_{L^1(t_k,t_{k+1})}\le b_k\,\de^2\left(\exp(\ka\,\de^{-\f\al{2-\al}})(\underbrace{\log(\cdots|\log}_kD|\cdots))^{-2(\f2\al-1)}+1\right),\]
where $b_k$ is exactly \eqref{eq-def-bk}. As before, we choose
\begin{equation}\label{eq-def-de}
\de_{k+1}:=\left(\f{\ka\,\al}{2(2-\al)}\right)^{\f2\al-1}(\underbrace{\log(\cdots(\log|\log}_{k+1}D|)\cdots))^{-(\f2\al-1)}
\end{equation}
to balance the right-hand side of the above inequality, which yields
\[\|\rho\|_{L^1(t_k,t_{k+1})}\le2b_k\,\de_{k+1}^2=2b_k\left(\f{\ka\,\al}{2(2-\al)}\right)^{2(\f2\al-1)}(\underbrace{\log(\cdots(\log|\log}_{k+1}D|)\cdots))^{-2(\f2\al-1)}\]
or equivalently \eqref{eq-est-7}. In addition, the combination of \eqref{eq-est-7} with $k$ and \eqref{eq-est-7'} with $k-1$ implies \eqref{eq-est-7'} with $k$ immediately.

As a final step, we take $k=N'$ in \eqref{eq-est-7'} and recall $t_{N'+1}=T-\de_*$ to conclude
\[\|\rho\|_{L^1(0,T-\de_*)}\le C_2(\underbrace{\log(\cdots(\log|\log}_{N'+1}D|)\cdots))^{-2(\f2\al-1)},\]
where
\[C_2:=2\left(\f{\ka\,\al}{2(2-\al)}\right)^{2(\f2\al-1)}\sum_{j=0}^kb_j.\]
Meanwhile, now it is clear that the choice of $\de_*$ should be \eqref{eq-def-de} with $k=N'$, that is,
\[\de_*:=\left(\f{\ka\,\al}{2(2-\al)}\right)^{\f2\al-1}(\underbrace{\log(\cdots(\log|\log}_{N'+1}D|)\cdots))^{-(\f2\al-1)}.\]
Consequently, noting the fact that $N'\le N$, we arrive at the desired estimate \eqref{eq-stab-d} and the proof is completed.
\end{proof}

\Section{Numerical Reconstruction Method}\label{sec-recon}

This section is devoted to the establishment of an iteration method for the numerical treatment of Problem \ref{prob-isp}. Henceforth, we will only deal with problem \eqref{eq-ibvp-u} in a bounded domain $\Om$ with $0<\al<1$ for technical convenience. We denote the true solution of Problem \ref{prob-isp} by $\rho_*$, and write $u_*(t):=u_{\rho_*}(x_0,t)$ as the noiseless observation data. The choice of $x_0\in\Om$ can be arbitrary, i.e., we do not require $x_0\in\supp\,g$.

First we mention the coincidence of the current problem with the classical deconvolution problem, which is readily seen by taking $x=x_0$ in \eqref{eq-Duhamel-J}:
\begin{equation}\label{eq-Duhamel-J'}
J^{1-\al}u_\rho(x_0,t)=\int_0^t\rho(t-s)\,v(x_0,s)\,\rd s.
\end{equation}
Indeed, the deconvolution appears frequently as a standard example of ill-posed problems (see, e.g., \cite{KS05}). However, regardless of the plenary amount of existing approaches to the deconvolution problem, here we attempt to develop a specialized method taking advantage of the underlying fractional diffusion equation.

Practically, we are only given the discrete measurement data of $u_*$ and have to work on the discretized setting. For simplicity, we divide the time interval $[0,T]$ into an equidistant partition $0=s_0<s_1<\cdots<s_L=T$, i.e., $s_\ell=\f TL\ell$ ($\ell=0,1,\ldots,L$). Then the reconstruction of $\rho_*$ reduces to the determination of a vector $(\rho_*(s_0),\rho_*(s_1),\ldots,\rho_*(s_L))\in\BR^{L+1}$. As a constraint on this vector, we assume
\[\rho_*(s_0)=0,\quad|\rho_*(s_\ell)|\le M,\ \ell=1,2,\ldots,L,\]
where $M>0$ is a given constant. In other words, we require that the true solution $\rho_*$ is bounded and vanishes at the origin, which is reasonable in practice. In order to prove the convergence of the iteration method proposed later, we shall interpret $\rho_*$ as an infinitely differentiable function. More precisely, we assume
\begin{equation}\label{eq-asp-rho*}
\rho_*\in C^\infty[0,T],\quad\rho_*^{(m)}(0)=0,\ \|\rho_*^{(m)}\|_{C[0,T]}\le M\left(\f{8L}T\right)^m,\ \forall\,m=0,1,2,\ldots.
\end{equation}
This is always possible by applying appropriate interpolation methods to the discrete pairs $\{(s_\ell,\rho_*(s_\ell))\}_{\ell=0}^L$. Correspondingly, the resulting solution $u_{\rho_*}$ is also smooth and especially $u_*=u_{\rho_*}(x_0,\,\cdot\,)\in C^\infty[0,T]$. On the other hand, since the known spatial component $g$ is also given on grid points, we can similarly regard $g$ as a sufficiently smooth function, so that the solution $v$ to the homogeneous problem \eqref{eq-ibvp-v} is continuous. This, together with the non-negativity of $g$, gives a constant $K>0$ such that
\begin{equation}\label{eq-asp-v}
0\le v(x_0,t)\le K,\quad0\le t\le T.
\end{equation}

Now we are well prepared to propose the fixed point iteration
\begin{equation}\label{eq-def-itr}
\rho_m=\left\{\!\begin{alignedat}{2}
& 0, & \quad & m=0,\\
& \rho_{m-1}+\f{\pa_t^\al(u_*-u_{\rho_{m-1}}(x_0,\,\cdot\,))}K, & \quad & m=1,2,\ldots,
\end{alignedat}\right.
\end{equation}
where $K>0$ refers to the same constant in \eqref{eq-asp-v}. The convergence of \eqref{eq-def-itr} is guaranteed by the following theorem.

\begin{thm}\label{thm-cov}
Under the assumptions \eqref{eq-asp-rho*} and \eqref{eq-asp-v}, the sequence $\{\rho_m\}_{m=0}^\infty$ generated by the iteration \eqref{eq-def-itr} converges uniformly to the true solution $\rho_*$ in $C[0,T]$, i.e.,
\[\lim_{m\to\infty}\|\rho_m-\rho_*\|_{C[0,T]}=0.\]
\end{thm}

\begin{proof}
Let the assumptions \eqref{eq-asp-rho*} and \eqref{eq-asp-v} be satisfied. First it is obvious that for $\rho\in C^1[0,T]$ with $\rho(0)=0$, there holds
\begin{equation}\label{eq-rep-Caputo}
\pa_t^\al u_\rho(x_0,t)=\int_0^t\rho'(s)\,v(x_0,t-s)\,\rd s.
\end{equation}
Actually, taking the usual time derivative in \eqref{eq-Duhamel-J'}, by definition and $\rho(0)=0$ we have
\[D_t^\al u_\rho(x_0,t)=\f\rd{\rd t}J^{1-\al}u_\rho(x_0,t)=\int_0^t\rho'(t-s)\,v(x_0,s)\,\rd s.\]
Due to the homogeneous initial condition of $u_\rho$, Lemma \ref{lem-Caputo-RL} indicates $\pa_t^\al u_\rho(x_0,t)=D_t^\al u_\rho(x_0,t)$ and thus \eqref{eq-rep-Caputo} immediately.

On the other hand, investigating the residue $\rho_*-\rho_m$ produced by \eqref{eq-def-itr}, formally we can write
\[\rho_*-\rho_m=\left(I-\f1K\pa_t^\al\right)^m\rho_*,\quad\forall\,m=0,1,2,\ldots.\]
We shall show by induction that
\begin{equation}\label{eq-residue}
(\rho_*-\rho_m)(t)=\int_0^t\rho_*^{(m)}(s)\,\Phi_m(t-s)\,\rd s,\quad m=1,2,\ldots,
\end{equation}
where
\begin{equation}\label{eq-def-Phi}
\Phi_m(t):=\left\{\!\begin{alignedat}{2}
& 1-\f1Kv(x_0,t), & \quad & m=1,\\
& \int_0^t\Phi_{m-1}(t-s)\left(1-\f1Kv(x_0,s)\right)\rd s, & \quad & m=2,3,\ldots.
\end{alignedat}\right.
\end{equation}
We start from $m=1$. Thanks to the assumption \eqref{eq-asp-rho*}, we can employ \eqref{eq-rep-Caputo} to deduce
\[(\rho_*-\rho_1)(t)=\int_0^t\rho_*'(s)\,\rd s-\f1K\int_0^t\rho_*'(s)\,v(x_0,t-s)\,\rd s=\int_0^t\rho_*'(s)\,\Phi_1(t-s)\,\rd s.\]
Now assume that \eqref{eq-residue} holds true for some $m=1,2,\ldots$, namely,
\[(\rho_*-\rho_m)(t)=\left(I-\f1K\pa_t^\al\right)^m\rho_*(t)=\int_0^t\rho_*^{(m)}(s)\,\Phi_m(t-s)\,\rd s=\int_0^t\rho_*^{(m)}(t-s)\,\Phi_m(s)\,\rd s.\]
We proceed to the case of $m+1$. By the assumption $\rho_*^{(m)}(0)=0$, we calculate
\[(\rho_*-\rho_m)'(t)=\int_0^t\rho_*^{(m+1)}(t-s)\,\Phi_m(s)\,\rd s=\int_0^t\rho_*^{(m+1)}(s)\,\Phi_m(t-s)\,\rd s.\]
As $(\rho_*-\rho_m)(0)=0$, again we take advantage of \eqref{eq-rep-Caputo} to obtain
\[\pa_t^\al(\rho_*-\rho_m)(t)=\int_0^tv(x_0,t-s)\int_0^s\rho_*^{(m+1)}(\tau)\,\Phi_m(s-\tau)\,\rd\tau\rd s.\]
Since $\rho_*^{(m)}(0)=0$ also gives $\rho_*^{(m)}(s)=\int_0^s\rho_*^{(m+1)}(\tau)\,\rd\tau$ for any $s>0$, we further calculate
\begin{align}
(\rho_*-\rho_{m+1})(t) & =\left(I-\f1K\pa_t^\al\right)\left(I-\f1K\pa_t^\al\right)^m\rho_*(t)=\left(I-\f1K\pa_t^\al\right)(\rho_*-\rho_m)(t)\nonumber\\
& =\int_0^t\left(\int_0^s\rho_*^{(m+1)}(\tau)\,\rd\tau\right)\Phi_m(t-s)\,\rd s\nonumber\\
& \quad\,-\f1K\int_0^tv(x_0,t-s)\int_0^s\rho_*^{(m+1)}(\tau)\,\Phi_m(s-\tau)\,\rd\tau\rd s\label{eq-cov-1}\\
& =\int_0^t\rho_*^{(m+1)}(\tau)\int_\tau^t\left(\Phi_m(t-s)-\f1Kv(x_0,t-s)\,\Phi_m(s-\tau)\right)\rd s\rd\tau\label{eq-cov-2}\\
& =\int_0^t\rho_*^{(m+1)}(\tau)\int_\tau^t\Phi_m(s-\tau)\left(1-\f1Kv(x_0,t-s)\right)\rd s\rd\tau\nonumber\\
& =\int_0^t\rho_*^{(m+1)}(\tau)\,\Phi_{m+1}(t-\tau)\,\rd\tau\nonumber,
\end{align}
where $\Phi_{m+1}$ takes the exact form of \eqref{eq-def-Phi} with $m+1$. Here we changed the order of integration in \eqref{eq-cov-1}, and applied the obvious identity
\[\int_\tau^t\Phi_m(t-s)\,\rd s=\int_\tau^t\Phi_m(s-\tau)\,\rd s\]
in \eqref{eq-cov-2}. Consequently, the claim \eqref{eq-residue} holds for all $m=1,2,\ldots$. Meanwhile, the assumption \eqref{eq-asp-v} implies
\[0\le1-\f1Kv(x_0,t)=\Phi_1(t)\le1,\quad0\le t\le T.\]
According to the inductive relation \eqref{eq-def-Phi}, it is straightforward to obtain
\[0\le\Phi_m(t)\le\f{t^{m-1}}{(m-1)!},\quad\forall\,m=1,2,\ldots.\]
Applying the above inequality and \eqref{eq-asp-rho*} to \eqref{eq-residue}, for any $t\in[0,T]$ we estimate
\begin{align*}
|(\rho_*-\rho_m)(t)| & \le\int_0^t|\rho_*^{(m)}(s)|\,\Phi_m(t-s)\,\rd s\le M\left(\f{8L}T\right)^m\int_0^t\f{s^{m-1}}{(m-1)!}\,\rd s\\
& =M\left(\f{8L}T\right)^m\f{t^m}{m!}\le\f{M(8L)^m}{m!}\to0\quad(m\to\infty).
\end{align*}
Therefore, we conclude $\|\rho_*-\rho_m\|_{C[0,T]}\to0$ as $m\to\infty$ and the proof is completed.
\end{proof}

In practice, we implement the iteration \eqref{eq-def-itr} in the following way. First, we solve the homogeneous problem \eqref{eq-ibvp-v} with the given spatial component $g$. This not only provides us with the data of $v(x_0,t)$, but also gives the upper bound $K$ in \eqref{eq-asp-v} which will be used repeatedly in the iteration. Next, from the observation data $u_*=u_{\rho_*}(x_0,\,\cdot\,)$ we calculate its Caputo derivative $\pa_t^\al u_*$. Since $\rho_0=0$ implies $\rho_1=\f1K\pa_t^\al u_*$, we can start the iteration from $m=1$ directly. Owing to the representation \eqref{eq-rep-Caputo}, eventually we can rewrite \eqref{eq-def-itr} as
\begin{equation}\label{eq-itr}
\rho_{m+1}(t)=\left\{\!\begin{alignedat}{2}
& \f1K\pa_t^\al u_*(t), & \quad & m=0,\\
& (\rho_1+\rho_m)(t)-\f1K\int_0^t\rho_m'(s)\,v(x_0,t-s)\,\rd s, & \quad & m=1,2,\ldots.
\end{alignedat}\right.
\end{equation}
Then it becomes obvious that at each iterative update, the main computational costs only involve the numerical differentiation of $\rho_{m-1}$ and its convolution with $v(x_0,\,\cdot\,)$. Therefore, throughout the reconstruction procedure, it suffices to solve only one fractional diffusion equation \eqref{eq-ibvp-v} offline, and the remaining iteration is accomplished by repeating one-dimensional data manipulations. As a result, one can expect that the proposed iterative update \eqref{eq-itr} is extremely efficient. The treatment for the ill-posedness in the numerical differentiation will be explained in Section \ref{sec-numer}.

We conclude this section with stating the main algorithm for the numerical reconstruction of Problem \ref{prob-isp}.

\begin{algo}\label{algo-itr}
Let the spatial component $g(x)$ ($x\in\Om$), the observation point $x_0\in\Om$ and the observation data $u_*(t)=u_{\rho_*}(x_0,t)$ ($0\le t\le T$) be given. Fix the stopping criteria $\ve>0$ and set $m=0$, $\rho_0=0$.
\begin{enumerate}
\item Solve the homogeneous problem \eqref{eq-ibvp-v} to obtain $v(x_0,t)$ and fix a constant $K>0$ satisfying \eqref{eq-asp-v}.
\item Compute $\rho_{m+1}$ by the iterative update \eqref{eq-itr}.
\item If $\|\rho_{m+1}-\rho_m\|_{L^2(0,T)}\le\ve$, then stop the iteration. Otherwise, update $m\leftarrow m+1$ and return to Step 2.
\end{enumerate}
\end{algo}

\Section{Numerical Experiments}\label{sec-numer}

In this section, we implement Algorithm \ref{algo-itr} proposed above to evaluate its numerical performance by several examples.

We start from the general settings of the numerical reconstruction. Throughout this section, basically we fix the fractional order $\al=0.9$ unless specified otherwise. Since both unknown function and observation data are variables in time, the spatial dimension is unimportant, so that it suffices to consider $d=1$ and even take $\Om=(0,1)$ without loss of generality. We set $T=1$ and choose the spatial component of the source term as
\[g(x)=\left\{\!\begin{alignedat}{2}
& \sin\left(2\pi x-\f\pi2\right), & \quad & \f14<x<\f34,\\
& 0, & \quad & \mbox{else}.
\end{alignedat}\right.\]
Then it is readily seen that $g\in H_0^1(0,1)$. Especially, we take the observation point as $x_0=\f18$ so that $x_0\not\in\supp\,g$. In this case, we can integrate by parts to further simplify \eqref{eq-itr} as
\begin{equation}\label{eq-itr'}
\rho_{m+1}(t)=\left\{\!\begin{alignedat}{2}
& \f1K\pa_t^\al u_*(t), & \quad & m=0,\\
& (\rho_1+\rho_m)(t)-\f1K\int_0^t\rho_m(t-s)\,\pa_sv(x_0,s)\,\rd s, & \quad & m=1,2,\ldots,
\end{alignedat}\right.
\end{equation}
which even circumvents the numerical differentiation at each step. Note that in this case, $\pa_tv(x_0,\,\cdot\,)$ makes sense in $L^1(0,T)$ owing to Lemma \ref{lem-ibvp}(c). On the other hand, we set the stopping criteria $\ve=10^{-5}$ in the algorithm, and denotes the reconstruction result as $\ov\rho$.

In the numerical experiments, we test Algorithm \ref{algo-itr} with two choices of true solutions to Problem \ref{prob-isp}, namely, a smooth one
\begin{equation}\label{eq-true1}
\rho_*(t)=\sin2\pi t+10\,t,
\end{equation}
and a non-smooth one
\begin{equation}\label{eq-true2}
\rho_*(t)=\left\{\!\begin{alignedat}{2}
& 3t, & \quad & 0\le t\le\f13,\\
& 1, & \quad & \f13<t<\f23,\\
& 3t-1, & \quad & \f23\le t\le1.
\end{alignedat}\right.
\end{equation}
To solve the forward problems for the observation data $u_*(t)$ and the solution $v$ to the homogeneous problem \eqref{eq-ibvp-v}, we discretize the space-time region $\ov\Om\times[0,T]=[0,1]^2$ into $64\times128$ equidistant meshes, and utilize the L1 time stepping method (see \cite{JinLazarovZhou:L1,LinXu:2007}). After computing the value of $v(x_0,t)$, we choose $K=0.2$ according to \eqref{eq-asp-v}.\medskip

First we evaluate the numerical performance of Algorithm \ref{algo-itr} with the noiseless data $u_*$ generated by the true solutions. In this case, we do not need any special treatments for the numerical differentiation involved in \eqref{eq-itr} or \eqref{eq-itr'}. The numerical results for the true solutions defined in \eqref{eq-true1} and \eqref{eq-true2} are displayed in Figure \ref{fig-noiseless}, which are obtained from the iteration \eqref{eq-itr} after $1348$ and $1305$ iterations respectively. For the smooth solution \eqref{eq-true1}, Figure \ref{fig-noiseless} (left) shows the convergence of the sequence $\{\rho_m\}_{m=0}^\infty$, which confirm Theorem \ref{thm-cov}. Furthermore, we attempt the true solution \eqref{eq-true2} which does not belong to $C^\infty[0,T]$. However, Figure \ref{fig-noiseless} (right) illustrates that Algorithm \ref{algo-itr} still works, which means the smoothness restriction on $\rho$ can be suitably weakened in the numerical reconstruction.
\begin{figure}[htbp]\centering
\includegraphics[trim=3cm 2mm 3cm 1cm,clip=true,width=.47\textwidth]{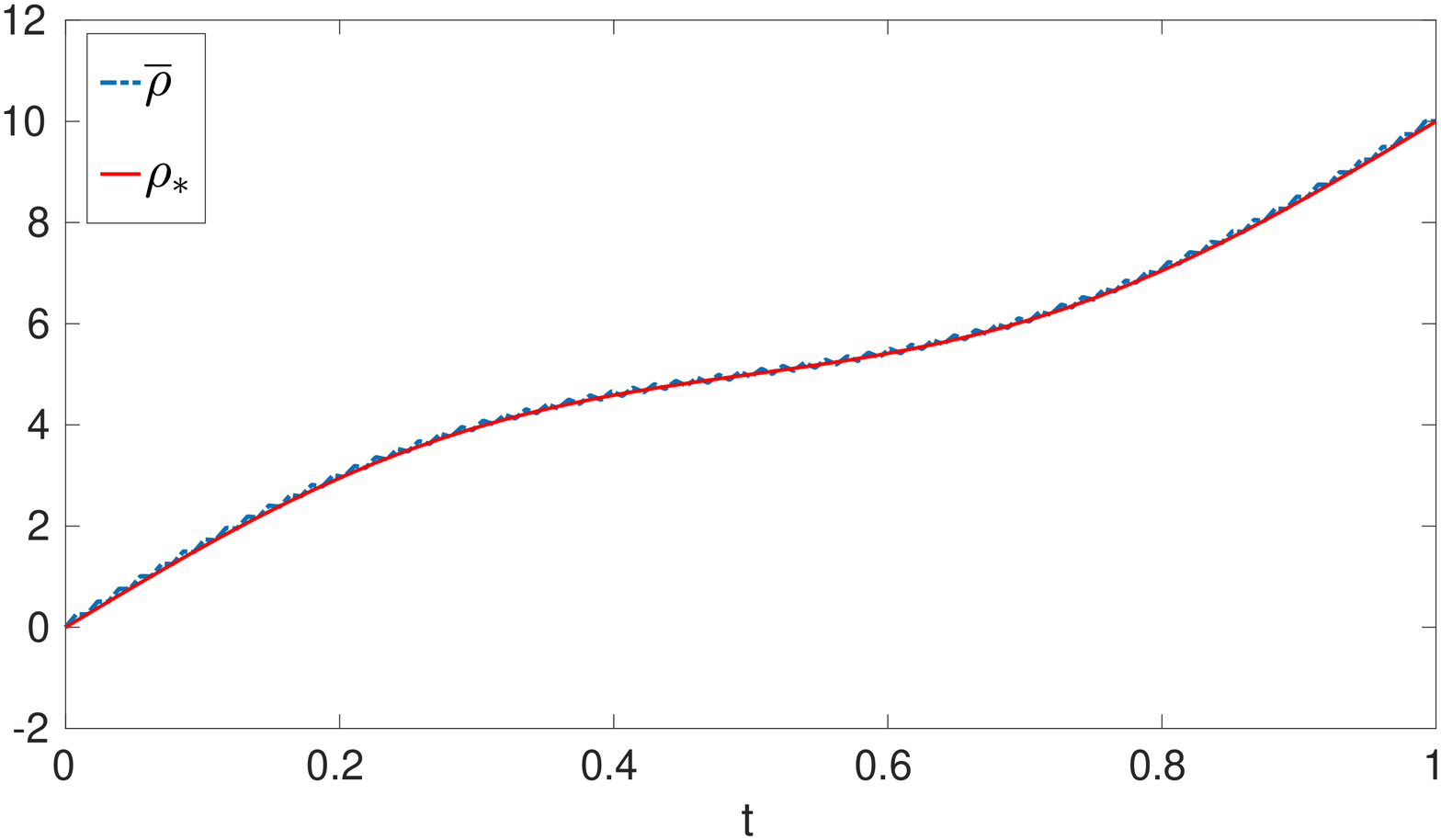}\qquad
\includegraphics[trim=3cm 2mm 3cm 1cm,clip=true,width=.47\textwidth]{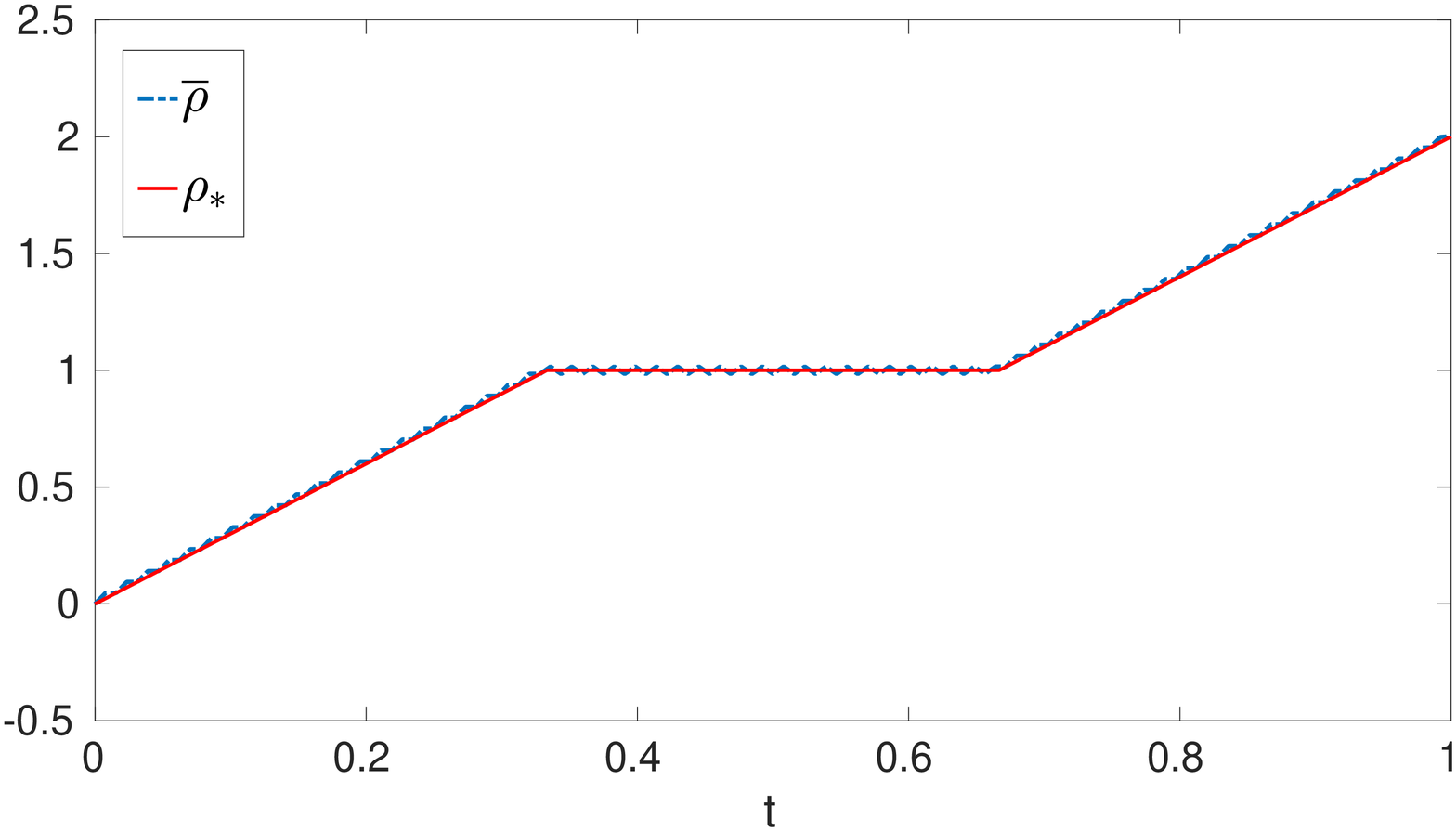}\\
\caption{True solutions $\rho_*$ and their numerical reconstructions $\ov\rho$ obtained by Algorithm \ref{algo-itr} with noiseless data. Left: The smooth true solution in \eqref{eq-true1}. Right: The non-smooth true solution in \eqref{eq-true2}.}\label{fig-noiseless}
\end{figure}

With the noiseless data, it turns out from Figure \ref{fig-noiseless} that Algorithm \ref{algo-itr} produces surprisingly accurate approximations to the respective true solutions which are almost indistinguishable from each other. However, though we do not add any artificial noise, there should exist certain error in the numerical solution of the forward problems. In view of the multiple logarithmic stability asserted in Theorem \ref{thm-stab-log}, such tiny error in the observation data may result in tremendous error in the reconstruction. Therefore, the above numerical results suggest the possibility to improve the theoretical stability of Problem \ref{prob-isp}. On the other hand, we also observe an interesting monotonicity property in the iteration. More precisely, the sequences $\{\rho_m\}_{m=0}^\infty$ generated by Algorithm \ref{algo-itr} are always monotonically increasing and converge to $\rho_*$ in both cases of \eqref{eq-true1} and \eqref{eq-true2}. This phenomenon is demonstrated by Figure \ref{fig-monotone}, where we plot the first four steps in the iteration of both cases with different fractional orders $\al$. If one can rigorously show the monotonicity of $\{\rho_m\}$ under certain conditions, then the accuracy of numerical results can be explained by the Lipschitz stability stated in Theorem \ref{thm-stab}(a), because $\rho_{m+1}-\rho_m$ ($m=0,1,\ldots$) does not change sign. Unfortunately, we could not prove such a monotonicity property, and the problem still remains open.
\begin{figure}[htbp]\centering
\includegraphics[trim=3cm 5mm 35mm 15mm,clip=true,width=.47\textwidth]{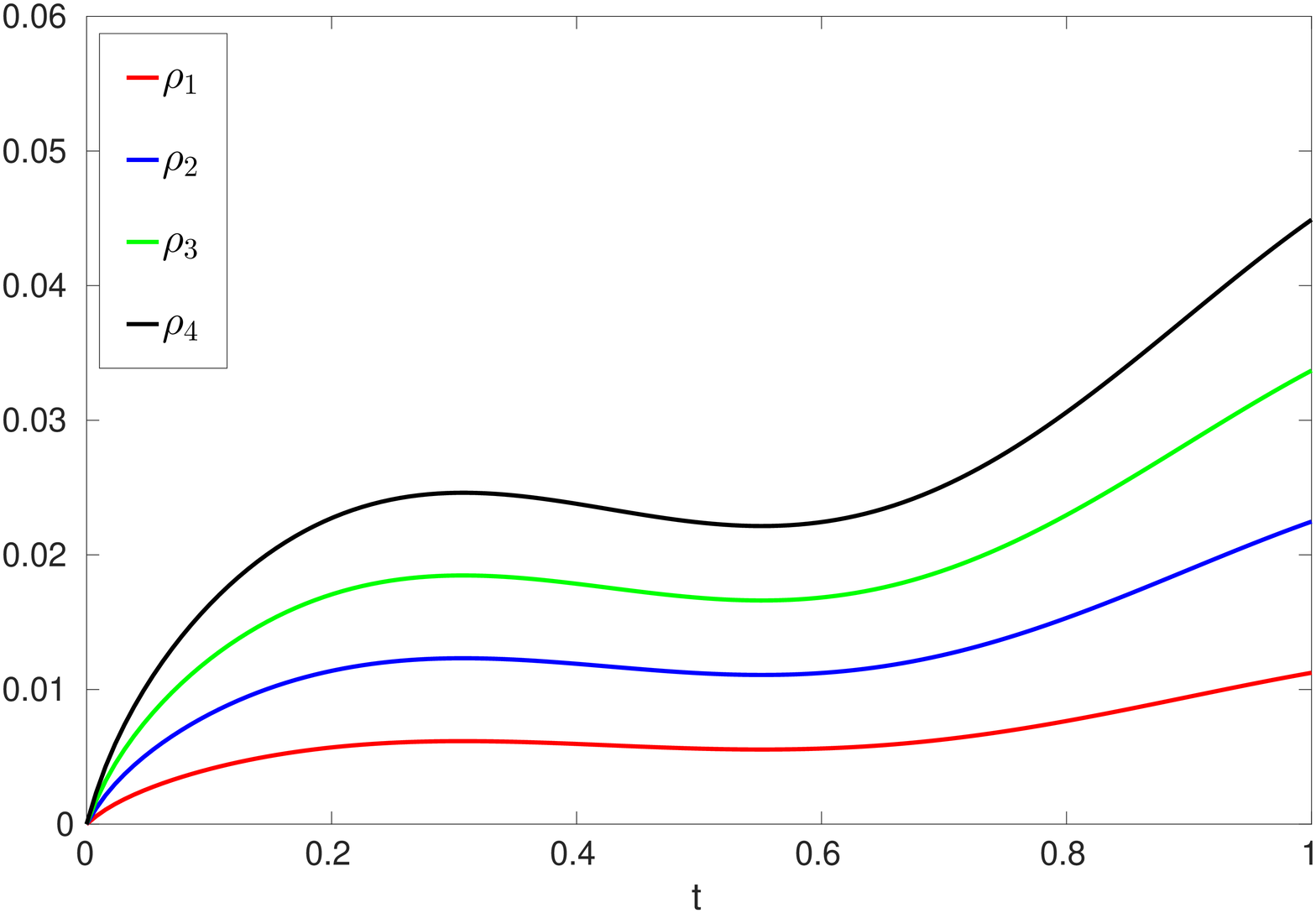}\qquad
\includegraphics[trim=25mm 5mm 35mm 15mm,clip=true,width=.47\textwidth]{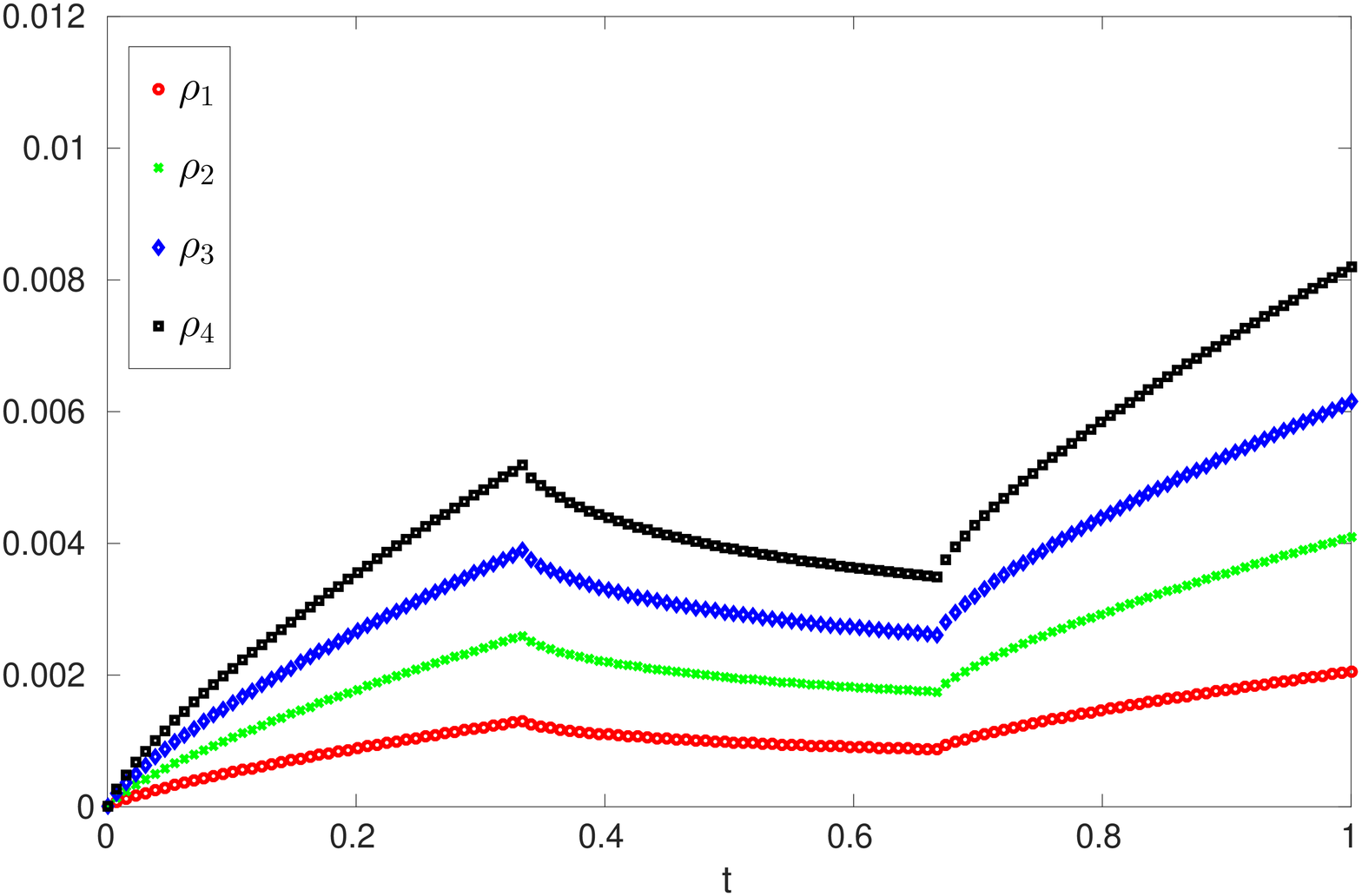}\\
\caption{The first four steps in the iteration generated by Algorithm \ref{algo-itr} with noiseless data. Left: The smooth true solution in \eqref{eq-true1} with $\al=0.5$. Right: The non-smooth true solution in \eqref{eq-true2} with $\al=0.3$.}\label{fig-monotone}
\end{figure}

Next, we proceed to deal with the noisy data for the practical purpose. Although the theoretical stability was discussed in an $L^1(0,T)$ framework, here for simplicity we just add uniform random noises to the noiseless data to produce the noisy data as
\[w_\si(t)=u_*(t)+\si\|u_*\|_{C[0,T]}\,\mathrm{rand}(-1,1),\]
where $\si>0$ is the relative noise level and $\mathrm{rand}(-1,1)$ denotes the random number uniformly distributed in $[-1,1]$. According to the iterative update \eqref{eq-itr}, the oscillation in $w_\si$ will definitely cause huge error in computing its Caputo derivative to obtain $\rho_1$, and further influence the subsequent numerical differentiation appearing in the iteration. To this end, we have to slightly reform Algorithm \ref{algo-itr} in the following way. For $f\in C[0,T]$ and $\ep\in(0,T)$, define the extension $\wt f$ of $f$ on $[-\ep,T+\ep]$ as
\[\wt f(t):=\begin{cases}
2\,f(0)-f(-t), & -\ep\le t<0,\\
f(t), & 0\le t\le T,\\
2\,f(T)-f(2T-t), & T<t\le T+\ep.
\end{cases}\]
Then we can introduce the mollification
\[f^\ep(t):=\int_{t-\ep}^{t+\ep}\ze_\ep(t-s)\,\wt f(s)\,\rd s,\quad0\le t\le T,\]
where the mollifier $\ze_\ep$ is defined by
\[\ze_\ep(t):=\left\{\!\begin{alignedat}{2}
& \f{15}{16\ep}\left(1+\f t\ep\right)^2\left(1-\f t\ep\right)^2, & \quad & |t|\le\ep,\\
& 0, & \quad & |t|>\ep.
\end{alignedat}\right.\]
Now it is obviously that $\ze_\ep\ge0$, $\int_{-\infty}^\infty \ze_\ep(t)\,\rd t=1$ and $f^\ve\in C^1[0,T]$. Replacing $u_*$ and $\rho_m$ by $w_\si^\ep$ and $\rho_m^\ep$ respectively in \eqref{eq-itr}, we arrive at the regularized iteration
\begin{equation}\label{eq-itr-reg}
\rho^\epsilon_{m+1}(t)=\left\{\!\begin{alignedat}{2}
& \f1K\pa_t^\al w_\si^\ep(t), & \quad & m=0,\\
& (\rho_1^\epsilon+\rho_m^\epsilon)(t)-\f1K\int_0^t(\rho_m^\ep)'(s)\,v(x_0,t-s)\,\rd s, & \quad & m=1,2,\ldots.
\end{alignedat}\right.
\end{equation}

Substituting \eqref{eq-itr} with \eqref{eq-itr-reg} in Algorithm \ref{algo-itr}, again we evaluate the numerical performance with the true solutions \eqref{eq-true1} and \eqref{eq-true2} and different noise levels $\si$. Here we take $\ep=\f5{128}$ in the mollifier $\ze_\ep$, where $\f1{128}$ is the step length in time in the numerical discretization. The comparisons of true solutions with their reconstructions are shown in Figure \ref{fig-noise}. In all examples, the regularized iteration \eqref{eq-itr-reg} terminates within $40$ steps.
\begin{figure}[htbp]\centering
\includegraphics[trim=3cm 2mm 3cm 1cm,clip=true,width=.47\textwidth]{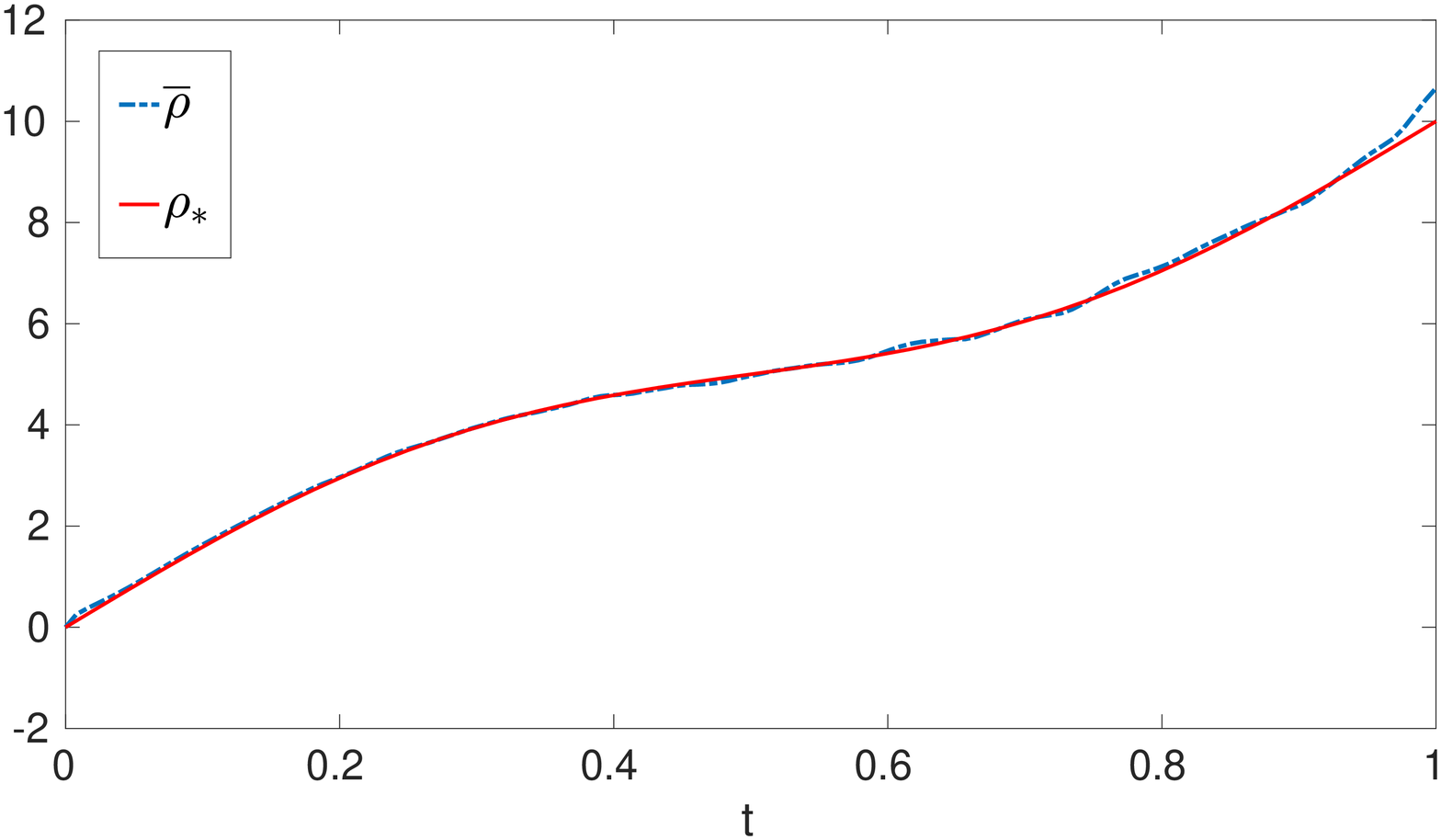}
	\qquad
\includegraphics[trim=3cm 2mm 3cm 1cm,clip=true,width=.47\textwidth]{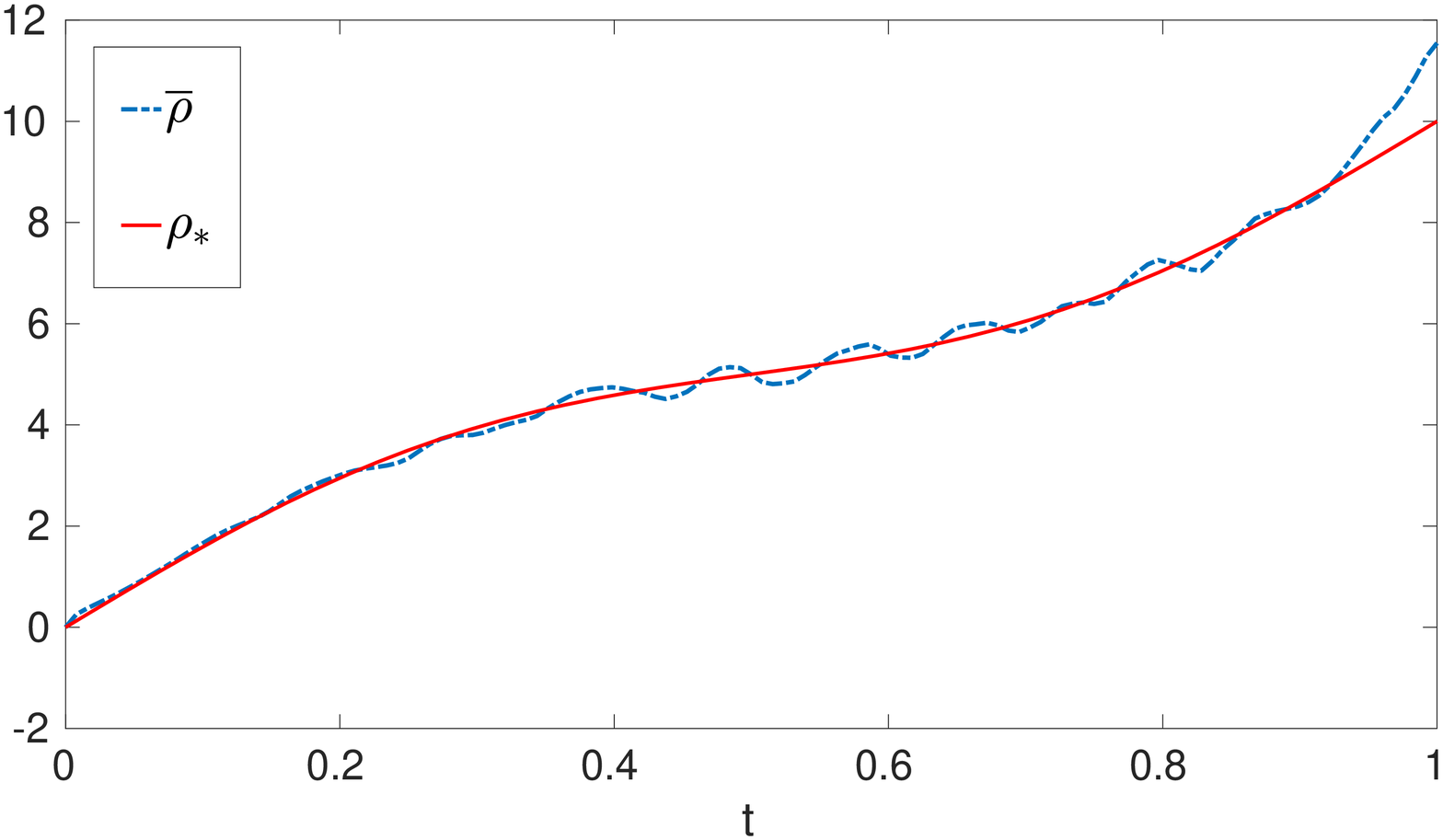}	\\[5mm]
\includegraphics[trim=3cm 2mm 3cm 1cm,clip=true,width=.47\textwidth]{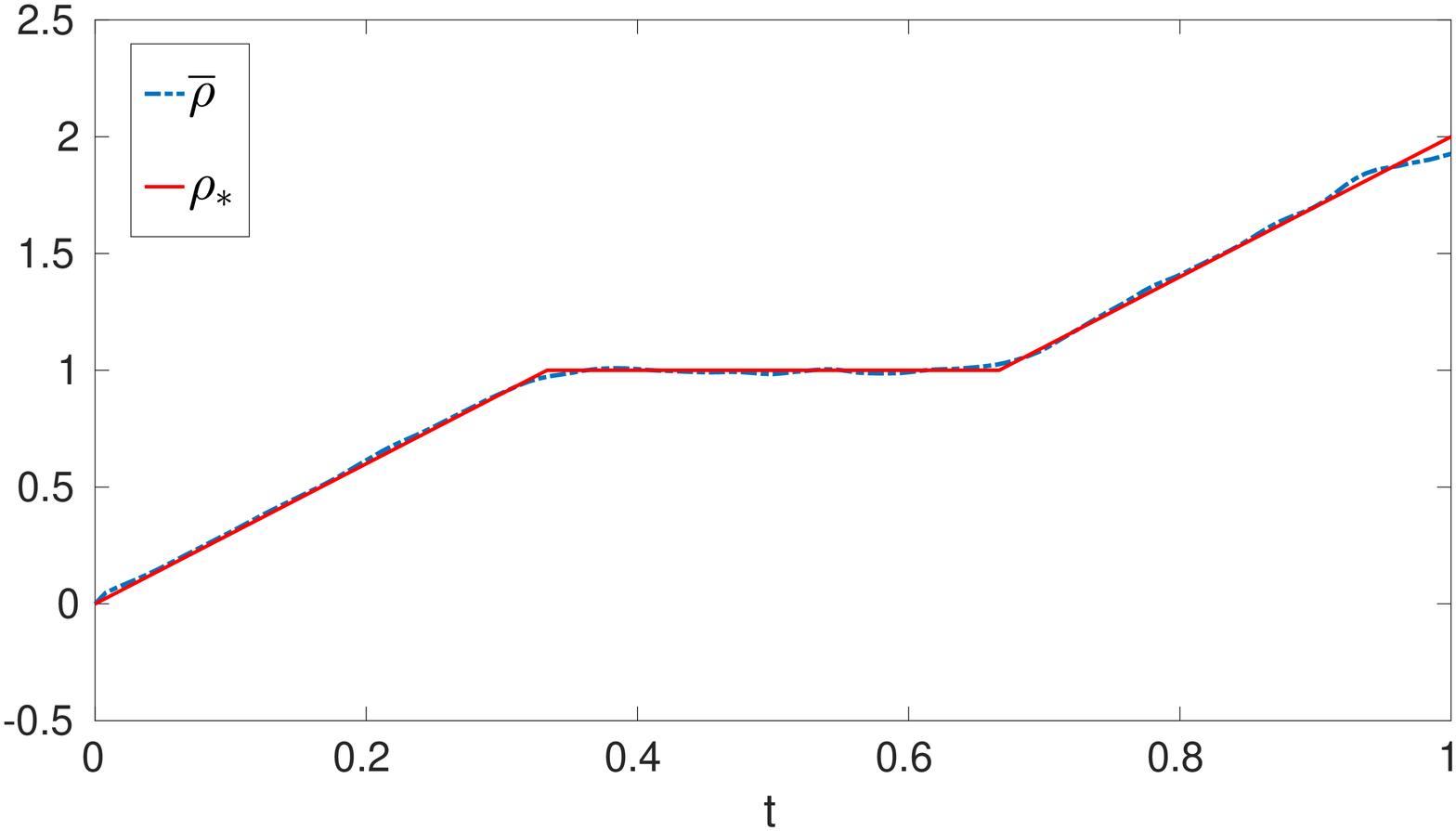}\qquad
\includegraphics[trim=3cm 2mm 3cm 1cm,clip=true,width=.47\textwidth]{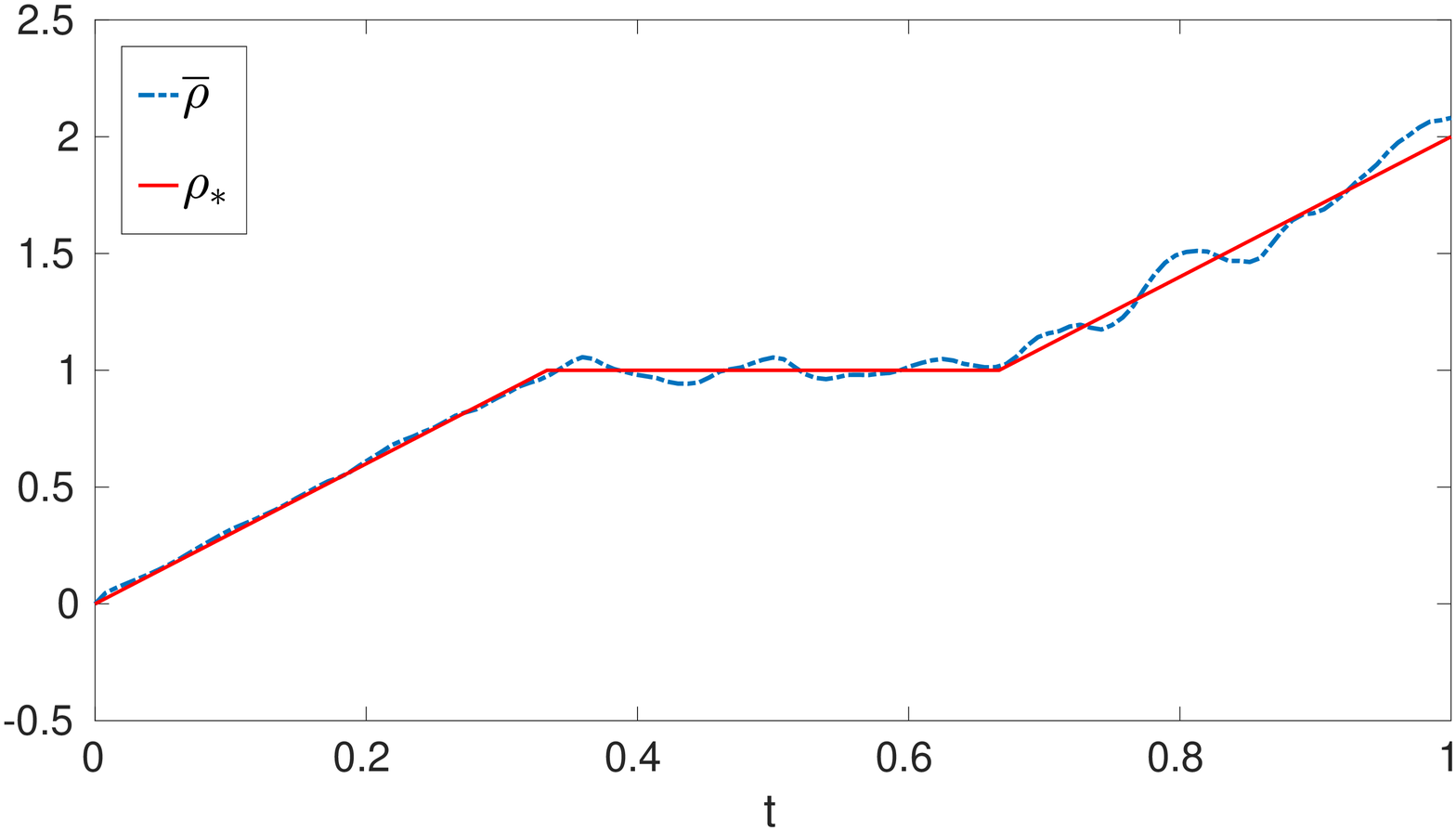}
\\
\caption{True solutions $\rho_*$ and their numerical reconstructions $\ov\rho$ obtained by the regularized iteration \eqref{eq-itr-reg} with noisy data. Upper row: The smooth true solution in \eqref{eq-true1}. Bottom row: The non-smooth true solution in \eqref{eq-true2}. Left column: Noise level $\si=1\%$. Right column: Noise level $\si=5\%$.}\label{fig-noise}
\end{figure}

It reveals from Figure \ref{fig-noise} that the simple mollification method effectively reduces the ill-posedness of the numerical differentiation in the iteration. Moreover, we can observe that all of the reconstructed solutions $\ov\rho(t)$ match the true solutions more accurately for small $t$. In the rear part of the time, the error tends to increase with respect to $t$. This phenomenon provides evidence for the necessity of an extra observation interval $(T-\de,T)$ in Theorem \ref{thm-stab}, which indicates the increasing instability near $T$.

\Section{Conclusions}\label{sec-concl}

From both theoretical and numerical aspects, in this paper we study the inverse problem on determining $\rho(t)$ in the (time-fractional) diffusion equation
\[(\pa_t^\al-\tri)u(x,t)=\rho(t)g(x)\quad(0<\al\le1)\]
by the observation taken along $\{x_0\}\times(0,T)$. In existing literature, the majority only dealt with the case of $x_0\in\supp\,g$, in which the problem turned out to be moderately ill-posed. If $x_0\not\in\supp\,g$, it remains a long-standing open problem even for $\al=1$ in spite of its practical significance. Motivated by \cite{STY02}, we restrict sign changes of $\rho$ to be finite, and require an extra interval of observation data. Based on the reverse convolution inequality and the lower bound of positive solutions, we prove stability results for the problem in Theorems \ref{thm-stab} and \ref{thm-stab-log}. Especially, we establish the multiple logarithmic stability for the case of $x_0\not\in\supp\,g$, which, as far as the authors know, seems to be the first affirmative answer to the problem regardless of its weakness.

Numerically, we develop a fixed point iteration for the reconstruction (see Algorithm \ref{algo-itr}), whose convergence is guaranteed by Theorem \ref{thm-cov}. In order to treat noisy data, we apply the mollification method to stabilize the numerical differentiation involved in the iteration. The efficiency and accuracy of our approach are demonstrated by several numerical experiments.

As a representative, we only consider the simplest (time-fractional) parabolic operator $\pa_t^\al-\tri$ in this paper, but most probably our argument also works for more general formulations, e.g., replacing $\tri$ by a uniformly non-degenerate elliptic operator. To this end, we shall investigate the corresponding fundamental solutions to give lower estimates for positive solutions. Other future topics include the improvement of theoretical stability under certain assumptions as well as the monotonicity property of the proposed iteration method.

\bigskip
%%%%%%%%%%%%%%%%%%%%%%%%%%%%%%%%%%%%%%%%%%%%%%%%%

{\bf Acknowledgement}\ \ This work is supported by Grant-in-Aid for Scientific Research (S) 15H05740, Japan Society for the Promotion of Science (JSPS). The first author is supported by JSPS Postdoctoral Fellowship for Overseas Researchers, Grant-in-Aid for JSPS Fellows 16F16319 and A3 Foresight Program ``Modeling and Computation of Applied Inverse Problems'', JSPS. The second author is supported by NSF Grant DMS-1620138.

\bibliographystyle{abbrv}
\bibliography{inverse_source}

\end{document}